\documentclass[12pt]{article}
\usepackage{amsmath}
\usepackage{amssymb}
\usepackage{amsthm}
\usepackage{mathrsfs}
\usepackage[left=20mm,right=20mm,top=25mm,bottom=25mm]{geometry}
\usepackage[dvipdfmx]{graphicx}
\usepackage{hyperref}
\usepackage{color}
\usepackage{mathtools}
\usepackage{bm}

\mathtoolsset{showonlyrefs}

\newtheorem{prop}{Proposition}[section]
\newtheorem{thm}[prop]{Theorem}

\newtheorem{ddef}[prop]{Definition}
\theoremstyle{remark}

\DeclareMathAlphabet{\mathpzc}{OT1}{pzc}{m}{it}

\DeclareMathOperator{\md}{d}

\newcommand{\re}{{\mathbb R}}
\newcommand{\ce}{{\mathbb C}}
\newcommand{\ze}{{\mathbb Z}}

\newcommand{\mi}{\mathrm{i}}

\begin{document}
\title{Numerical Analysis for the Plateau Problem by the Method of Fundamental Solutions
}
\makeatletter
\renewcommand\@date{{%
  \vspace{-\baselineskip}%
  \large\centering
  \begin{tabular}{@{}c@{}}
    Koya Sakakibara\textsuperscript{1,2} \\
    \normalsize ksakaki@ous.ac.jp
  \end{tabular}%
  \quad and\quad
  \begin{tabular}{@{}c@{}}
    Yuuki Shimizu\textsuperscript{3} \\
    \normalsize yshimizu@g.ecc.u-tokyo.ac.jp
  \end{tabular}

  \bigskip

  \textsuperscript{1}Department of Applied Mathematics, Faculty of Science, Okayama University of Science\par
  \textsuperscript{2}RIKEN iTHEMS\par
  \textsuperscript{3}Graduate School of Mathematical Sciences, The University of Tokyo
}}

\maketitle
\abstract{
Towards identifying the number of minimal surfaces sharing the same boundary from the geometry of the boundary, 
we propose a numerical scheme with high speed and high accuracy. 
Our numerical scheme is based on the method of fundamental solutions. 
We establish the convergence analysis for Dirichlet energy and $L^\infty$-error analysis for mean curvature. 
Each of the approximate solutions in our scheme is a smooth surface, which is a significant difference from previous studies that required mesh division.
}

\section{Introduction}
In 1762, Lagrange proposed the problem of finding a surface with the least area spanned by a given closed Jordan curve ~\cite{Lagrange_1760}.
Because the first variation of the area of a surface gives the mean curvature vector, the mean curvature of the surface with the least area is everywhere zero.
Thus, a surface whose mean curvature is everywhere zero is called a minimal surface. 
In a more general sense, Lagrange's problem can be rephrased as the problem of finding a minimal surface spanned by a given closed Jordan curve.
Later in 1873, Plateau investigated the properties of minimal surfaces through an experiment on soap films and pointed out that a single closed wire, regardless of its geometry, bounded at least one soap film.
Nowadays, the problem of the existence of minimal surfaces bounded by a given closed Jordan curve is called Plateau's problem.

Although Douglas has established the existence of a solution to the Plateau problem~\cite{Douglas1931solution}, it is difficult to determine whether the solution is unique or whether there are finitely many solutions, even if the simple closed curve is smooth and if the solutions are homeomorphic to the disk. 
The problem has not been completely solved. 
For example, the solution is unique if a closed Jordan curve has a one-to-one parallel projection onto a convex closed curve in the plane~\cite{Rado_1930} or if the total curvature is less than or equal to $4\pi$~\cite{Nitsche1973unique}.
For the case where there is more than one solution, it is known that there are two other minimal surfaces with the same boundary as the Enneper surface, which is one of the exact solutions of the Plateau problem ~\cite{Nitsche1989book}.
The finite solvability is established for generic curves~\cite{BohmeTromba1981}, stable solutions~\cite{Koiso1983}, and a curve whose total curvature is less than $6\pi$~\cite{Nitsche1978finite}. 
As seen from these results, the problem of determining the number of solutions according to the geometry of a closed Jordan curve is one of the most important issues in the geometrical analysis of minimal surfaces.

To solve this problem, numerical analysis on minimal surfaces should help to find a connection between the geometry of the curve and the number of minimal surfaces sharing the curve as the boundary. 
Since minimal surfaces are considered stationary points of functionals, such as area and Dirichlet integrals,
we can obtain minimal surfaces as the convergence limit of some optimization problem. 
Therefore, one possible method to determine the number of solutions corresponding to a given simple closed curve is to take a random initial guess for the optimization problem and count the number of different minimal surfaces obtained as the convergence limit of the optimization problem.

In order to make the above approach practical, it is necessary to establish a numerical scheme, which is
\begin{enumerate}
	\item so fast that many initial values can be handled in a reasonable time, and 
	\item so accurate that the numerical solutions obtained can be distinguished.
\end{enumerate}
Therefore, this paper aims to establish a numerical scheme with high speed and high accuracy enough to achieve the above goals.

Numerical analysis of minimal surfaces has a long history, and since Douglas~\cite{douglas1927method} proposed a numerical scheme using the finite difference method, various methods have been proposed up to the present day.
Tsuchiya gave the first theoretical convergence analysis~\cite{tsuchiya1987discrete,tsuchiya1990note}.
He considered the finite element method and proved the existence of discrete minimal surfaces and convergence for the $H^1$ norm.
However, because he used non-direct arguments, he did not derive the convergence order.
The finite element method is used for the analysis in the Dziuk--Huthinson paper~\cite{dziuk1999discrete-numerics,dziuk1999discrete-convergence}, as in Tsuchiya's papers.
Although the target is somewhat limited to nondegenerate minimal surfaces, the error for the $H^1$ norm is proved to be $\mathrm{O}(h)$.
Their result is the first analytical one that includes the order of convergence.
Following this work, Pozzi~\cite{pozzi2004estimate} gave the $L^2$ error estimate for finite element solutions.
In Dziuk--Hutchinson~\cite{dziuk2006finite}, a method for computing surfaces with a specified mean curvature is given with error analysis.
Numerical analysis of minimal surfaces whose boundaries are given as polygons is considered by Hinze~\cite{hinze1996numerical}, who proves convergence for the $H^1$ norm.
Other than these, various numerical methods have been proposed, for example, high-precision spatial discretization using B-spline curves~\cite{hao2013approximation} and higher-order polynomials~\cite{trasdahl2011high}, computation of mean curvature flow by finite volume method~\cite{tomek2019discrete}, improved performance by mesh refinement~\cite{grodet2018finite}, and methods based on algebraic topology~\cite{schumacher2019variational} and differential forms~\cite{wang2021computing}.
However, to our knowledge, no convergence analysis has been given for any of them.

As seen in Definition \ref{def:min_surf}, numerical computation of minimal surfaces requires solving the Dirichlet boundary value problem for the Laplace equation on the unit disk. 
Therefore, a fast and accurate solver for the Laplace equation can be used to compute minimal surfaces more accurately than in previous studies. 
Thus, in this paper, we employ the method of fundamental solutions (MFS for short), a numerical method for the potential problem.
The MFS is a mesh-free numerical solver because it does not require meshing of the region, as with the finite element and finite difference methods. 
It is so named because a linear combination of the fundamental solution of the partial differential operator of interest constructs the approximate solution. 
Under certain conditions, the approximate solution by the MFS has the remarkable property of exponentially converging to the exact solution with an increasing number of approximation points (Theorem \ref{thm:MFS_error_estimate} holds in the setting of this paper). 
Moreover, since a linear combination of fundamental solutions constructs the approximate solution, the approximate solution is smooth, and differential operations can be performed analytically, resulting in a smooth approximation of minimal surfaces.
When the Jordan curve is embedded into a two-dimensional plane, the problem of finding a minimal surface is reduced to the problem of finding a conformal map from the unit disk onto the simply-connected region bounded by the Jordan curve. Amano and his colleagues have been known to compute conformal mappings with high accuracy using the MFS~\cite{amano1991bidirectional,amano1994charge,amano1998charge,amano2012numerical,sakakibara2020bidirectional}.
However, to our knowledge, the MFS has no successful application to the numerical computation of minimal surfaces.

By utilizing the convergence theorem for the Dirichlet boundary value problem of the Laplace equation in the disk region (Theorem \ref{thm:MFS_error_estimate}), we succeed in giving the existence of approximate minimal surfaces (Theorem \ref{thm:existence_approx_surf}), and the convergence of Dirichlet energy and evaluation of mean curvature of approximate minimal surfaces (Theorem \ref{thm:Dirichlet_energy_mean_curvature}). 
These are the first results of the MFS numerical analysis of minimal surfaces.
As a result, the MFS can now obtain smooth approximations of minimal surfaces, a significant difference from previous studies that required mesh division.

This paper is organized as follows. 
In Section \ref{sec:preliminaries}, we first briefly summarize the geometry of surfaces, then formulate the Plateau problem and present known results on the existence of minimal surfaces. 
Then, the MFS is explained in detail, and the results of the convergence analysis used in this paper are presented. In Section \ref{sec:scheme_basis}, we build the basis of an algorithm for solving the Plateau problem and, in particular, present the existence theorem for approximate minimal surfaces. 
Section \ref{sec:error_analysis} shows that the approximate minimal surfaces obtained by our algorithm converge to the minimal surfaces. In particular, we show the Dirichlet energy's convergence and the mean curvature's $L^\infty$-error. 
In Section \ref{sec:numerics}, after describing the algorithm in actual numerical computations, we demonstrate the usefulness of the algorithm proposed in this paper by giving various numerical examples. 
In Section \ref{sec:search_all}, we propose an algorithm for computing all minimal surfaces based on the numerical scheme proposed in this paper and verify its usefulness through numerical experiments.
Finally, Section \ref{sec:concluding_remarks} summarizes the paper and provides directions for future research.

\section{Preliminaries}
\label{sec:preliminaries}

\subsection{Geometry of surfaces}

In this section, we briefly describe some basic notions from differential geometry closely related to this work.  
Let $B_\rho$ be the disk with the radius $\rho$ in the complex plane $ z=x^1+\mi x^2\in  \ce$ and $\partial B$ be its boundary: 
\begin{align}
	B_\rho:&=\{ z\in\ce |\,| z|<\rho\},\\
	\partial B_{\rho}:&=\{ z\in\ce |\,| z|=\rho\}. 
\end{align}
For notational simplicity, we write $B=B_1$. 
Let us denote the differentiation by $\partial=\partial/\partial  z$, $\bar\partial=\partial/\partial  z$, $\partial_1= \partial/\partial x^1$, and $\partial_2 = \partial/\partial x^2$, respectively. 
For each $f\in C^1(\ce)$, we now have 
\begin{align}
	\partial
	&=(\partial_1-\mi\partial_2)/2,\\
	\bar\partial &=(\partial_1+\mi\partial_2)/2,\\
	\triangle
	&= 4\partial\bar\partial=\partial_u^2 +\partial_v^2,\\
	4|\partial f|^2
	&=(\partial_1 f)^2+(\partial_2 f)^2,\\
	4(\partial f)^2
	&= (\partial_1 f)^2-(\partial_2 f)^2+2\mi (\partial_1 f)(\partial_2 f),
\end{align}
Given a vector-valued function $X: z\in  B\to (X_ i( z) )_{ i=1}^3\in\re^3$, we write $\partial_jX=(\partial_j X_i)_{i=1}^3\in\re^3$ and $\triangle X=(\triangle X_i)_{i=1}^3\in\re^3$. 

A surface $(M,g)$ is a Riemann surface equipped with a Riemannian metric $g$, which yields complex analysis, and Riemannian geometry can be utilized as the surface. 
We now focus on some geometric structures defined on the surface. 
Remenber that a complex function $f\colon\ce\to \ce$ is said to be holomorphic if it satisfies the Cauchy-Riemann equation
\begin{align}
	\bar\partial f=0.
\end{align}
Moreover, if $f$ is bijective and the inverse is also holomorphic, then $f$ is called a biholomorphism. 

The Riemann surface $M$ is equipped with complex charts ${(U_\alpha,\phi_\alpha)}_\alpha$ such that for each $(U_\alpha,\phi_\alpha)$ and $(U_\beta,\phi_\beta)$, $\phi_\alpha\circ \phi_\beta^{-1} :\phi_\beta( U_\alpha \cap U_\beta) \to \phi_\alpha(U_\alpha\cap U_{\beta}) \subset \ce $ is a biholomorphism, and $M=\bigcup_\alpha U_\alpha$ holds. 
The union of all compatible complex charts defines the complex structure on the surface, and a smooth $2$-manifold with a complex structure is called a Riemann surface. 
Given two Riemann surfaces $M$ and $N$, a diffeomorphism $f:M\to N$ is a biholomorphism if $f$ is a biholomorphism in each complex chart. 
In addition, if a biholomorphism exists between two Riemann surfaces $M$ and $N$, then the Riemann surface $M$ is said to be biholomorphic to $N$. 

The Riemannian metric $g$ induces a conformal structure. 
A chart $(U,\phi)$ is isothermal if the metric has the following local representation in the chart:
\begin{align}
	g= \lambda^2 |\md z|^2,
\end{align}
where $\lambda>0$ and $\lambda$ is called the conformal factor. 
A geometric structure defined by the union of all compatible isothermal charts is called a conformal structure, and a smooth $2$-manifold equipped with a conformal structure is called a conformal surface. 
Given two conformal surfaces $M$ and $N$, a diffeomorphism $f:M\to N$ is a conformal mapping if $f$ is conformal in each isothermal chart. 
Moreover, if a conformal mapping exists between two Riemann surfaces $M$ and $N$, then the Riemann surface $M$ is said to be conformal to $N$. 
It is easy to see that every isothermal chart becomes complex, and the converse holds using conventional identification $z=x^1+\mi x^2$. 
Hence, given surfaces $M$ and $N$, a diffeomorphism $f:M\to N$ is a conformal mapping if and only if $f$ is a biholomorphism. 

For instance, let a vector-valued function $X\colon B\to  \re^3$ be a $C^1$ immersion; that is, the the rank of the matrix $(\partial_i X_j)_{i,j}$ is $2$.
Then, disk $B$ may have two different Riemannian metrics: first, the standard 2D Euclidean metric $g_s=(\md x^1)^2+(\md x^2)^2$, and second, the induced metric $g_X= \sum_{i,j} \langle \partial_i X, \partial_j X\rangle \md x^i\md x^j$, where $\langle\cdotp,\cdotp\rangle$ is the 3D Euclidean metric. 
For the identity map ${\rm id}:(B,g_s)\to (B,g_X)$ to be conformal, the immersion $X$ must satisfy
\begin{align}
	\langle\partial_1 X, \partial_1 X\rangle=\langle  \partial_2 X, \partial_2 X\rangle,\quad 
	 \langle \partial_1 X, \partial_2 X\rangle=0.
\end{align}
By using the complex coordinate, it is rewritten as 
\begin{align}
	\Phi_X(z):&= \sum_{i=1}^3 (\partial X_i(z))^2\\
	&=\frac14 \left(\langle\partial_1 X, \partial_1 X\rangle_z-\langle  \partial_2 X, \partial_2 X\rangle_z +2\mi  
	 \langle \partial_1 X, \partial_2 X\rangle_z\right)\\
	&=0.
\end{align}
The function $\Phi_X: B\to \ce$ is called the complex dilatation. 
If and only if the complex dilatation $\Phi_X$ vanishes, the induced metric $g_X$ satisfies $g_X=\lambda^2 g_s$, where $\lambda^2= \langle\partial_1 X, \partial_1 X\rangle=\langle  \partial_2 X, \partial_2 X\rangle$. 
Moreover, if $\triangle X=0$, $\Phi_X$ is a holomorphic function since 
\begin{align}
	\bar\partial \Phi_X 
	&=\sum_{i=1}^3 2 \partial X_i(z)\bar\partial \partial X_i(z) \\
	&=\frac{1}{2}\sum_{i=1}^3 \partial X_i(z)\triangle X_i(z) \\
	&=0.
\end{align}

\subsection{Plateau problem}
We review the Plateau problem in terms of a construction of solutions. 
The following facts in this subsection is briefly introduced by the Dierkes, Hildebrandt and Sauvigny~\cite{Dierkes_Hildebrandt_Sauvigny_2010} in Chapter 4.2. 
The Plateau problem is formulated as finding an immersed surface spanned by a given Jordan curve that minimizes the area. 
A variational argument yields that the minimizer of the area has zero mean curvature. 
The mean curvature $H_X$ of a parametrized surface $X:B\to \re^3$ satisfies 
\begin{align}
	H_X n = \triangle_X X,
\end{align}
where $n=\partial_1 X\times \partial_2 X/|\partial_1 X\times \partial_2 X| $ is the unit normal vector and $\triangle_X$ is the Laplace-Beltrami operator associated with the metric $g_X$ if $|\partial_1 X\times \partial_2 X|\ne 0$. 
In particular, if the complex dilatation $\Phi_X$ vanishes, we can deduce 
\begin{align}
	\triangle_X= \lambda^{-2} \triangle,
\end{align}
where $\triangle$ is the Laplacian $\triangle=\partial_1^2+\partial_2^2$. 
Hence, the minimal surface is formulated as a solution to the following problem.
\begin{ddef}
\label{def:min_surf}
	Given a closed Jordan curve $\Gamma\subset\re^3$, $X: \bar B\to 
	\re^3 $ is called a minimal surface spanned by $\Gamma$ if the following four conditions are satisfied:
	\begin{enumerate}
		\item $X\in C^0(\bar B,\re^3)\cap C^2(B,\re^3);$
		\item $\triangle X=0;$
		\item the restriction $X|_{\partial B}: \partial B\to \Gamma$ is a homeomorphism; 
		\item $\Phi_X=0.$
	\end{enumerate}
\end{ddef}

Courant~\cite{Courant_1937} established the existence of a minimal surface spanned by a given closed Jordan curve by minimizing the Dirichlet energy 
\begin{align}
	D(X):= \frac12 \int_B (|\partial_1 X|^2+|\partial_2 X|^2) \md x^1\md x^2
\end{align}
in the Sobolev class $X\in H^1(B,\re^3)$. 
More precisely, the minimizing problem of the Dirichlet integral is performed in the following space of admissible functions. 
\begin{ddef}
	Given a closed Jordan curve $\Gamma$ in $\re^3$, a mapping $X\in H_2^1(B,\re^3)$ is said to be of class $\mathcal{C}(\Gamma)$ for a fixed orientation if its Sobolev trace $X|_{\partial B}$ can be represented by a weakly monotonic, continuous mapping $\varphi:\partial B\to \Gamma$ onto $\Gamma$. 
\end{ddef}
Unless otherwise stated, hereafter, $\mathcal{C}(\Gamma)$ is always defined for a fixed orientation.
Consequently, the minimizing problem of the Dirichlet integral is formulated as follows: 
\begin{align}
	\mathcal{P}(\Gamma):\quad D(X)\to \min\quad \text{in the class } \mathcal{C}(\Gamma).
\end{align}
The existence of a solution to the problem $\mathcal{P}(\Gamma)$ is obtained when $\mathcal{C}(\Gamma)$ is nonempty. 
In particular, it is satisfied if $\Gamma$ is a closed Jordan curve of finite length. 
\begin{thm}[{\cite[Chapter 4.3, Theorem 1]{Dierkes_Hildebrandt_Sauvigny_2010}}]
\label{thm:exist_min}
	If $\mathcal{C}(\Gamma)$ is nonempty, then the minimizing problem $\mathcal{P}(\Gamma)$ has at least one solution, continuous on $\bar B$ and harmonic in $B$. 
	In particular, $\mathcal{P}(\Gamma)$ has such a solution for every rectifiable curve $\Gamma$. 
\end{thm}
It follows from Weyl's lemma that all minimizers $X\in{\cal C}(\Gamma)$ are actually, i.e., $\triangle X=0$. 
However, changing the coordinate, the minimizer is not in general harmonic again in the new coordinate since the Laplacian $\triangle$ and the Dirichlet energy $D(X)$ change as the coordinates are changed. 
Courant~\cite{Courant_1937} showed that a minimal surface is obtained by taking the variation of the Dirichlet energy by changing variables. 
\begin{thm}[{\cite[Chapter 4.5, Theorem 2]{Dierkes_Hildebrandt_Sauvigny_2010}}]
	Every solution $X$ of the variational problem $\mathcal{P}(\Gamma)$ is a minimal surface. 
\end{thm}

In particular, the following is helpful to show the existence of the limit for an approximate solution for a minimal surface. 

\begin{thm}[{\cite[Chapter 4.3, Theorem 3]{Dierkes_Hildebrandt_Sauvigny_2010}}]
\label{thm:exist_lim}
	Let $\{\Gamma_n\}_{n\geq 1}$ be a sequence of closed (oriented)  Jordan curves in $\re^3$, which converge in the sense of Fr\'echet to some closed (oriented) Jordan curve $\Gamma$. 
	Let $\{X_n\}_{n\geq1}\subset C^0(\bar B,\re^3)\cap C^2(B,\re^3)$ be  $\triangle X_n=0$ and $X_n(\partial B)=\Gamma_n$. 
	Then, there exists a subsequence $\{X_{n_p}\}_{p\geq1}$ and $X\subset C^0(\bar B,\re^3)\cap C^2(B,\re^3)$ with $\triangle X=0$ and $X(\partial B)=\Gamma$ such that $X_{n_p}\to X$ uniformly on $\bar B$ as $p\to\infty$. 
\end{thm}

The above facts is briefly introduced by the Dierkes, Hildebrandt and Sauvigny~\cite{Dierkes_Hildebrandt_Sauvigny_2010} in Chapter 4.

\subsection{Method of fundamental solutions}
\label{subsec:MFS}

The method of fundamental solutions (MFS for short) is a mesh-free numerical solver for linear partial differential equations such as the Laplace equation, the Helmholtz equation, and the biharmonic equation. 
Its idea is quite simple, and the algorithm is described for the Laplace equation, which is the subject of this paper.

Let $\Omega$ be a bounded region in $\mathbb{C}$ with smooth boundary $\partial\Omega$, and consider the Dirichlet boundary value problem for the Laplace equation in $\Omega$ with a given boundary data $f:\partial \Omega\to \re  $:
\begin{align}
    \begin{dcases*}
        \triangle u=0&in $\Omega$,\\
        u=f&on $\partial\Omega$.
    \end{dcases*}
\end{align}
The MFS constructs an approximate solution for this problem by the following procedure.
\begin{enumerate}
    \item Take $N\in\mathbb{N}$ and fix it. 
        Moreover, arrange $N$ points $\zeta_k$ ($k=1,2,\ldots,N$) ``suitably'' in $\mathbb{C}\setminus\overline{\Omega}$, which we call the singular points.
    \item Seek an approximate solution $u^{(N)}$ in the following form:
        \begin{align}
            u^{(N)}(z)=\sum_{k=1}^NQ_kG(z-\zeta_k),
            \label{eq:MFS_approx_sol}
        \end{align}
        where $G(z)=(2\pi)^{-1}\log|z|$ is the fundamental solution of the Laplace operator.
        Note that $u^{(N)}$ satisfies the Laplace equation exactly in $\Omega$ since the singular points $\{\zeta_k\}_{k=1}^N$ are outside $\Omega$.
    \item Determine coefficients $\{Q_k\}_{k=1}^N$ by the collocation method.
        Namely, choose $N$ points $z_j$ ($j=1,2,\ldots,N$) ``suitably'' on $\partial\Omega$, and impose the following conditions:
        \begin{align}
        \label{eq:boundary_condition}
            u^{(N)}(z_j)=f(z_j),
            \quad
            j=1,2,\ldots,N.
        \end{align}
\end{enumerate}

Eq.~\eqref{eq:boundary_condition} can be rewritten as a linear system called the collocation equations, 
\begin{align}
\label{eq:collocation_equation}
	\mathbf{G}\boldsymbol{Q}=\boldsymbol{f},
\end{align}
where 
\begin{align}
	\mathbf{G} = (G(z_j-\zeta_k))_{j,k}\in\mathbb{R}^{N\times N},
	\quad
    \boldsymbol{Q}=(Q_{k})_k\in\mathbb{R}^N,
    \quad
    \boldsymbol{f}=(f(z_j))_j\in\mathbb{R}^N. 
\end{align}
As seen from the algorithm, the MFS does not require meshing the region, and the approximate solution is constructed by choosing appropriate points on the boundary and outside the region.
However, what constitutes appropriate point placement is still mathematically unsolved.
In this paper, it is only necessary to consider the case where the problem region is the unit disk.
In this case, it is natural to place the collocation points $\{z_j\}_{j=1}^N$ and singular points $\{\zeta_k\}_{k=1}^N$ uniformly on concentric circles as follows:
\begin{alignat}{2}
    z_j&=\omega^j,&\quad&j=1,2,\ldots,N,\label{eq:collocation_points}\\
    \zeta_k&=R\omega^k,&\quad&k=1,2,\ldots,N,\label{eq:singular_points}
\end{alignat}
where $R>1$ and $\omega=\exp(2\pi\mathrm{i}/N)$.

Then, we can solve the collocation equations~\eqref{eq:collocation_equation} explicitly. 
Since the coefficient matrix $\boldsymbol{G}$ is now circulant, its inverse $\boldsymbol{G}^{-1}=(G^{-1}_{kj})_{kj}$ is presented by 
\begin{align}
    G_{kj}^{-1}
    \coloneqq
    \frac1N\sum_{l=1}^N\frac{\omega^{(k-j)(l-1)}}{\varphi_{l-1}^{(N)}},
\end{align}
where 
\begin{align}
    \varphi_p^{(N)}
    \coloneqq
    \sum_{k=1}^N\omega^{p(k-1)}G(1-\zeta_k),
    \quad
    p\in\mathbb{Z}.
\end{align}
As a result, the coefficients $Q_{k}$ are explicitly given by
\begin{align}
\label{eq:analytic_Q}
    Q_{k}=\sum_{j=1}^NG_{kj}^{-1}f(z_j)= \frac1N\sum_{j=1}^N\sum_{l=1}^N\frac{\omega^{(k-j)(l-1)}}{\varphi_{l-1}^{(N)}}f(z_j)
\end{align}
for each $k=1,2,\ldots,N$. 
Hence, we find that an approximate solution exists and that it can be concretely constructed.

Under the above setting, the following theorem holds.
\begin{thm}[{\cite[Theorem 2]{katsurada1988mathematical}, \cite[Theorem 2.3]{katsurada1989mathematical}}]
    \label{thm:MFS_error_estimate}
    \begin{enumerate}
        \item Suppose that the boundary data $f$ is real analytic.
        Then, there are constants $C>0$ and $\tau\in(0,1)$, independent of $N$, such that
        \begin{align}
            \|u-u^{(N)}\|_{L^\infty(B)}
            \le
            C\tau^N.
        \end{align}
        \item Let $\{f_n\}$ be the Fourier coefficients of $f$.
        \begin{enumerate}
            \item If the Fourier series $\sum_nf_n\mathrm{e}^{\mathrm{i}n\theta}$ is absolutely convergent, then the approximate solution uniformly converges to the exact solution in $B$ as $N\to\infty$; that is,
            \begin{align}
                \|u-u^{(N)}\|_{L^\infty(B)}
                \longrightarrow0
                \quad
                (N\to\infty)
            \end{align}
            \item If $f_n=\mathrm{O}(|n|^{-\alpha})$ for some $\alpha>1$ as $|n|\to\infty$, then we have
            \begin{align}
                \|u-u^{(N)}\|_{L^\infty(B)}=\mathrm{O}(N^{-\alpha+1})
                \quad
                (N\to\infty).
            \end{align}
        \end{enumerate}
    \end{enumerate}
\end{thm}

Since the approximate solution \eqref{eq:MFS_approx_sol} by the MFS is analytic, its derivatives can also be computed analytically.
For instance,
\begin{align}
    \partial u^{(N)}(z)
    =
    \sum_{k=1}^NQ_k\partial G(z-\zeta_k).
\end{align}
As pointed out in \cite[Section 4]{katsurada1988mathematical} and \cite[Remark 4.1]{katsurada1989mathematical}, under the same situation in Theorem \ref{thm:MFS_error_estimate}, we can also prove similar estimates for $\|\partial u-\partial u^{(N)}\|_{L^\infty(B)}$; the $W^{1,\infty}$-error $\|u-u^{(N)}\|_{W^{1,\infty}(B)}$ tends to $0$ as $N\to\infty$ under mild assumptions on the boundary data $f$.
The maximum principle for harmonic functions implies that
\begin{align}
    &\|u-u^{(N)}\|_{L^p(B)}
    \le
    \pi^{1/p}\|u-u^{(N)}\|_{L^\infty(B)},\\
    &\|\partial u-\partial u^{(N)}\|_{L^p(B)}
    \le
    \pi^{1/p}\|\partial u-\partial u^{(N)}\|_{L^\infty(B)}
\end{align}
for $p\in[1,\infty)$.
Repeating the same procedure, we find that, for $m\in\mathbb{N}$ and $p\in[1,\infty]$, the $W^{m,p}$-error $\|u-u^{(N)}\|_{W^{m,p}(B)}$ tends to $0$ as $N\to\infty$, depending on the regularity of the solution.

\section{Numerical scheme solving Plateau problem}
\label{sec:scheme_basis}
In what follows, we construct a minimal surface $X\in C^0(\bar B,\re^3)\cap C^2(B,\re^3)$ spanned by a given rectifiable closed Jordan curve $\Gamma\subset \re^3$. 
Let us fix a homeomorphism $b:\partial B\to \Gamma$.
Then, since $X|_{\partial B}: \partial B\to \Gamma$ is also homeomorphism, we deduce that there exists a homeomorphism $\phi:\partial B\to \partial B$ such that 
\begin{enumerate}
		\item $\triangle X=0$ on $B$;
		\item $X=b\circ \phi$ on $\partial B$;
		\item $\Phi_X=0$ on $B$,
\end{enumerate}
by taking $\phi= b^{-1}\circ X|_{\partial B}$. 
Note that for a given $\phi$, $X$ is solved using Poisson kernel $P$ of the Dirichlet boundary value problem; that is, $X$ is given by
\begin{align}
\label{eq:Poisson_kernel}
	X=X(z;\phi)=(P*(b_1\circ\phi),P*(b_2\circ\phi),P*(b_3\circ\phi)).
\end{align}
From this point of view, we can say $\phi$ is chosen to attain $\Phi_{X(\cdotp;\phi)}=0$. 
In other words, for a given boundary mapping $b:\partial B\to \Gamma$, find $\phi:\partial B\to \partial B$ subject to Eq.~\eqref{eq:Poisson_kernel} and $||\Phi_{X(\cdotp;\phi)}||_{L^\infty(B)}=0$. 
Remember that $\Phi_X$ is a holomorphic function if $\triangle X=0$. 
Owing to the maximum principle, for minimizing $||\Phi_{X(\cdotp;\phi)}||_{L^\infty(B)}$, it suffices to minimize $||\Phi_{X(\cdotp;\phi)}||_{L^\infty(\partial B_\rho )}$ for sufficiently close $\rho\in(0,1]$ to $1$. 
Hence, we can obtain the desired $\phi$ by solving the minimization problem $\min\{ ||\Phi_{X(\cdotp;\phi)}||_{L^\infty(\partial B_\rho )}|\,\phi\in \mathrm{Homeo}(\partial B)\}$, where $\mathrm{Homeo}(\partial B)$ is the space of all homeomorphisms on $\partial B$. 
In this paper, we solve Eq.~\eqref{eq:Poisson_kernel} and $\min\{ ||\Phi_{X(\cdotp;\phi)}||_{L^\infty(\partial B_\rho )}|\,\phi\in \mathrm{Homeo}(\partial B)\}$ by the MFS and Nesterov's accelerated gradient descent, respectively. 

First, discretize Eq.~\eqref{eq:Poisson_kernel} for a given $b:\partial B\to \Gamma$ and a given $\phi:\partial B\to \partial B$ by the following procedure. 
\begin{enumerate}
	\item Take $z_j$ and $\zeta_k$ given by Eq.~\eqref{eq:collocation_points} and Eq~\eqref{eq:singular_points} for a given $N\in \ze_{\ge1}$. 
	\item Take a transformation vector $\boldsymbol{\phi}=(\phi(z_j))_{j=1}^N\in\mathbb{T}^N$, where $\mathbb{T}^N$ is the $N$-dimensional torus.
	\item Seek an approximate solution $X_i^{(N)}(\cdotp ;\boldsymbol{\phi} ):\bar B\to \re $ for each $i=1,2,3$ in the following form: 
	 \begin{align}
	 X_i^{(N)}(z;\boldsymbol{\phi} )= \sum_{k=1}^N Q_{ik}(\boldsymbol{\phi} ) G(z-\zeta_k)\quad \text{on }B.
	 \end{align}
	\item Determine the coefficients $\boldsymbol{Q}_i(\boldsymbol{\phi})=(Q_{ik}(\boldsymbol{\phi}))_{k=1}^N \in \re^N$ by 
	\begin{align}
		X_i^{(N)}(z_j;\boldsymbol{\phi} )
		&= b_i(\phi_j),
	\end{align}
	which can be solved explicitly using Eq.~\eqref{eq:analytic_Q}.  
	\item Make an approximate surface $X^{(N)}(\cdotp ;\boldsymbol{\phi} ):\bar B\to \re^3$ by 
	\begin{align}
		X^{(N)}(z;\boldsymbol{\phi} )=(X_1^{(N)},X_2^{(N)},X_3^{(N)})
		(z;\boldsymbol{\phi} ).
	\end{align}
\end{enumerate}

Second, minimize $||\Phi_{X^{(N)}(\cdotp;\boldsymbol{\phi} )}||_{L^\infty(B)}$ in the class $\phi\in\mathrm{Homeo}(\partial B)$. 
Since $X^{(N)}$ is determined by $\boldsymbol{\phi}\in\mathbb{T}^N$, the admissible space for the minimization problem can be reduced to the finite-dimensional space $\mathbb{T}^N$ from the infinite-dimensional space $\mathrm{Homeo}(\partial B)$.
We write $\Phi_{X^{(N)}(\cdotp ;\boldsymbol{\phi} )}(z)$ by $\Phi^{(N)}(z;\boldsymbol{\phi})$ shortly. 
It is worth noting that $\triangle X^{(N)}(z,\boldsymbol{\phi} )=0$ for any $\boldsymbol{\phi}\in\mathbb{T}^N$, which yields that $\Phi^{(N)}(z;\boldsymbol{\phi})$ becomes a holomorphic function. 
Hence, it is sufficient to minimize $||\Phi^{(N)}(\cdotp ;\boldsymbol{\phi})||_{L^\infty(\partial B_\rho)}$ with a given sufficiently close $\rho\in(0,1]$ to $1$. 
In particular, we discretize $||\Phi^{(N)}(\cdotp ;\boldsymbol{\phi})||_{L^\infty(\partial B_\rho)}$ by 
\begin{align}
    E=E(\boldsymbol{\phi})=\sum_{j=1}^N|\Phi^{(N)}(\rho z_j;\boldsymbol{\phi})|^2.
\end{align}
Hence, we deduce to minimize $E(\boldsymbol{\phi})$ in the class $\boldsymbol{\phi}\in\mathbb{T}^N$ by the following procedure.  
\begin{enumerate}
	\item Choose an initial vector $\boldsymbol{\phi}_1\in \mathbb{T}^N$ arbitrarily and a small enough step size $\eta\in(0,1)$.
	\item Set $\boldsymbol{\varphi}_1=\boldsymbol{\phi}_1$
	\item Update 
	\begin{align}
		\boldsymbol{\varphi}_{n+1}
		&=\boldsymbol{\phi}_n-\eta \nabla E(\boldsymbol{\phi}_n),\\
		\boldsymbol{\phi}_{n+1}
		&= \boldsymbol{\varphi}_n+\frac{n-1}{n+2} \boldsymbol{\phi}_n.
	\end{align}
\end{enumerate}
We can compute $\nabla E$ analytically as follows.
The gradient of $E(\boldsymbol{\phi}) $ with respect to $\phi_j$ is given by
\begin{align}
    \partial_{\phi_j}E
    =
    2\sum_{l=1}^N\mathrm{Re}\left(\partial_{\phi_j}\Phi^{(N)}(\rho z_l;\boldsymbol{\phi})\overline{\Phi^{(N)}(\rho z_l;\boldsymbol{\phi})}\right),
\end{align}
and the term $\partial_{\phi_j}\Phi^{(N)}(\rho z_l;\boldsymbol{\phi})$ can be computed using
\begin{align}
    \partial_{\phi_j}\Phi^{(N)}(\rho z_l;\boldsymbol{\phi})
    =
    2\sum_{i=1}^3\partial X_i^{(N)}(\rho z_l;\boldsymbol{\phi})\partial_{\phi_j}\partial X_i^{(N)}(\rho z_l;\boldsymbol{\phi}).
\end{align}
Moreover, we have
\begin{align}
    \partial_{\phi_j}\partial X_i^{(N)}(\rho z_l;\boldsymbol{\phi})
    =
    \sum_{k=1}^N\partial_{\phi_j}Q_{ik}(\boldsymbol{\phi}) \partial G(\rho z_l-\zeta_k).
\end{align}
Owing to Eq.~\eqref{eq:analytic_Q}, $\partial_{\phi_j}Q_{ik}(\boldsymbol{\phi}) $ can be obtained analytically as
\begin{align}
    \partial_{\phi_j}Q_{ik}(\boldsymbol{\phi}) =G_{kj}^{-1}b_i'(\phi_j).
\end{align}
Hence, we can obtain $\nabla E$ explicitly. 
\color{black}


\section{Convergence and error analysis}
\label{sec:error_analysis}
We first show the existence of an approximate solution for the minimization problem with a given precision $\varepsilon>0$. 
\begin{thm}
    \label{thm:existence_approx_surf}
	Let $\Gamma \subset \re^3$ be a rectifiable closed Jordan curve with a fixed homeomorphism $b:\partial B\to \Gamma$. 
	Suppose the Fourier series of $b_i$ is absolutely convergent for each $i=1,2,3$. 
	Let $X^{(N)}(\cdotp ;\boldsymbol{\phi} ):\bar B\to \re^3$ be the approximate surface for a given $N\in\ze_{\ge1}$ and a given $\boldsymbol{\phi}\in\mathbb{T}^N$. 
	Let $\Phi^{(N)}(\cdotp ;\boldsymbol{\phi} ):\bar B\to \ce $ be the complex dilatation of $X^{(N)}$. 
	\begin{enumerate}
		\item For any $\varepsilon>0$ and any sufficiently large $N\in\ze_{\ge1}$, there exists $\boldsymbol{\varphi}=\boldsymbol{\varphi}(\varepsilon,N) \in\mathbb{T}^N$ such that 
			\begin{align}
			\|\Phi^{(N)}(\cdotp;\boldsymbol{\varphi}) \|_{L^\infty(B)}<\varepsilon.
			\end{align}
			We call $\boldsymbol{\varphi}$ \textit{$\varepsilon$-conformal configulation}.
		\item If the Fourier coefficient $b_{i,n}$ of $b_i$ satisfies $b_{i,n}=\mathrm{O}(|n|^{-\alpha})$ for each $i=1,2,3$ and some $\alpha>2$, then for any sufficiently large $N\in\ze_{\ge1}$, there exists $\boldsymbol{\varphi}=\boldsymbol{\varphi}(N) \in\mathbb{T}^N$ such that 
			\begin{align}
			\|\Phi^{(N)}(\cdotp;\boldsymbol{\varphi}) \|_{L^\infty(B)}=\mathrm{O}(N^{-\alpha+1})
                \quad
                (N\to\infty). 
			\end{align}
		 \item If $b_i$ is real analytic for each $i=1,2,3$, then there exists $\boldsymbol{\varphi}=\boldsymbol{\varphi}(N) \in\mathbb{T}^N$ and constants $C>0$ and $\tau\in(0,1)$, independent of $N$, such that 
			\begin{align}
			\|\Phi^{(N)}(\cdotp;\boldsymbol{\varphi}) \|_{L^\infty(B)}\le C\tau^N.  
			\end{align}
	\end{enumerate}
\end{thm}
\begin{proof}
	Take a minimal surface $X\in C^0(\bar B,\re^3)\cap C^2(B,\re^3)$ spanned by $\Gamma$ and define $\varphi:\partial B\to \partial B$ by $\varphi=b^{-1}\circ X|_{\partial B}$.
	In particular, we now have $\Phi_X=0$. 
	Setting $\boldsymbol{\varphi} =(\varphi(z_j))_{j=1}^N$, we show this is the desired object. 
	
	Since $X^{(N)}(\cdotp ;\boldsymbol{\varphi})$ is constructed by the MFS, we deduce from Theorem~\ref{thm:MFS_error_estimate} that for any $i=1,2,3$,
	\begin{align}
		\|\partial X_i-\partial X_i^{(N)}(\cdotp ;\boldsymbol{\varphi})\|_{L^\infty(B)}<1. 
	\end{align}
	Hence, we see that for each $z\in B$, 
	\begin{align}
		|\Phi^{(N)}(z;\boldsymbol{\varphi})|
		&=|\Phi_X(z)-\Phi^{(N)}(z;\boldsymbol{\varphi})|\\
		&\leq \sum_{i=1}^3 |\partial X_i(z)^2-\partial X_i^{(N)}(z;\boldsymbol{\varphi})^2| \\
		&\leq \sum_{i=1}^3 |\partial X_i(z)-\partial X_i^{(N)}(z;\boldsymbol{\varphi})|(2|\partial X_i(z)|+|\partial X_i(z)-\partial X_i^{(N)}(z;\boldsymbol{\varphi})|)\\
		&\leq \sum_{i=1}^3 \|\partial X_i-\partial X_i^{(N)}(\cdotp ;\boldsymbol{\varphi})\|_{L^\infty(B)}(2\|\partial X_i\|_{L^\infty(B)}+\|\partial X_i-\partial X_i^{(N)}(\cdotp ;\boldsymbol{\varphi})\|_{L^\infty(B)})\\
		&\leq C \|\partial X_i-\partial X_i^{(N)}(\cdotp ;\boldsymbol{\varphi})\|_{L^\infty(B)}. 
	\end{align}
	Hence, the decay of $\|\Phi^{(N)}(\cdotp;\boldsymbol{\varphi}) \|_{L^\infty(B)}$ follows directly from the decay of $\|\partial X_i-\partial X_i^{(N)}(\cdotp ;\boldsymbol{\varphi})\|_{L^\infty(B)}$ given in  Theorem~\ref{thm:MFS_error_estimate}. 
\end{proof}

We next see that the approximate surface for an $\varepsilon$-conformal configuration gives a minimal surface as $N\to\infty$. 
\begin{thm}
    \label{thm:Dirichlet_energy_mean_curvature}
	Let $\Gamma \subset \re^3$ be a rectifiable closed Jordan curve with a fixed homeomorphism $b:\partial B\to \Gamma$. 
	Suppose the Fourier series of $b_i$ is absolutely convergent for each $i=1,2,3$. 
	Let $\boldsymbol{\varphi} \in\mathbb{T}^N$ be an $\varepsilon$-conformal configuration for a given $\varepsilon>0$ and a sufficiently large $N\in\ze_{\ge1}$. 
	Let $X^{(N)}(\cdotp ;\boldsymbol{\varphi} ):\bar B\to \re^3$ be the approximate surface. 
	Then, there exists a minimal surface $X\in C^0(\bar B,\re^3)\cap C^2(B,\re^3)$ spanned by $\Gamma$ and a subsequence $\{X^{(N_p)}\}_{p\ge1}$ such that 
	\begin{align}
		\lim_{p\to \infty} D(X^{(N_p)})=D(X). 
	\end{align}
	Moreover, if $X$ is non-singular, i.e., $\det(\partial_i X\cdotp \partial_j X)\ne0$, there exists a constant $C>0$, independent of $p$ and $\varepsilon$, such that
	\begin{align}
		 \|H_{X^{(N_p)}}\|_{L^\infty(B)}<C\varepsilon. 
	\end{align}
\end{thm}
\begin{proof}
	Theorem~\ref{thm:MFS_error_estimate} gives the uniform convergence of $\Gamma_N=X^{(N)}(\partial B)$ to $\Gamma$. 
	Hence, applying Theorem~\ref{thm:exist_lim} to $\{X^{(N)}\}$, we obtain the existence of a subsequence $\{X^{(N_p)}\}_{p\ge1}$ and $X\subset C^0(\bar B,\re^3)\cap C^2(B,\re^3)$ with $\triangle X=0$ and $X(\partial B)=\Gamma$ such that $X^{(N_p)}\to X$ uniformly on $\bar B$ as $p\to\infty$. 
	Moreover, owing to the maximum principle, we see that there exists a constant $C>0$, independent $p$, such that 
	\begin{align}
		\| X- X^{(N_p)}\|_{W^{1,\infty}(B)}
		\le C \| X- X^{(N_p)}\|_{L^{\infty}(\partial B)}\to 0 
		\quad\text{as }p\to \infty. 
	\end{align}
	We thus deduce that 
	\begin{align}
		\|\Phi_X\|_{L^\infty(B)}
		&\le \|\Phi_X-\Phi \|_{L^\infty(B)}+\|\Phi_{X^{(N)}}\|_{L^\infty(B)}\\
		&\le C\| X- X^{(N_p)}\|_{W^{1,\infty}(B)}+\|\Phi_{X^{(N)}}\|_{L^\infty(B)}\\
		&\le 2\varepsilon.
	\end{align}
	Since $\varepsilon>0$ is arbitrary, we obtain $\|\Phi_X\|_{L^\infty(B)}=0$, which yields that $X$ is a minimal surface spanned by $\Gamma$. 
	We now have
	\begin{align}
		|D(X)-D(X^{(N_p)}|
		&\le\| X- X^{(N_p)}\|_{W^{1,2}(B)}\\
		&\le C\| X- X^{(N_p)}\|_{W^{1,\infty}(B)}\\
		&\le C \| X- X^{(N_p)}\|_{L^{\infty}(\partial B)}\to 0 
		\quad\text{as }p\to \infty.  
	\end{align}

	Lastly, we examine the convergence of the mean curvature.
	Define the first fundamental forms $g_{ij}^{(N)}$ as
    \begin{align}
        g_{ij}^{(N)}\coloneqq\langle\partial_iX^{(N)},\partial_jX^{(N)}\rangle,
        \quad
        i,j=1,2.
    \end{align}
    Since $\|\Phi_{X^{(N)}}\|_{L^\infty(B)}<\varepsilon$ holds, we have
    \begin{align}
        \|g_{11}^{(N)}-g_{22}^{(N)}\|_{L^\infty(B)}\le C\varepsilon,
        \quad
        \|g_{12}^{(N)}\|_{L^\infty(B)}\le C\varepsilon
    \end{align}
    for sufficiently large $N$.
    Here and hereafter, $C$ is a positive constant that can change with each appearance and does not depend on $N$.
    Since $g=(g_{ij})$ is non-singular, for sufficiently large $N$, we have
    \begin{align}
        \det g^{(N)}\ge C.
    \end{align}
    
    A pointwise error of the mean curvature is given by
    \begin{align}
        |H_{X}^{(N)}|
        &=
        |H_X-H_{X^{(N)}}|\\
        &=
        \left|
            \frac{g_{11}h_{22}+g_{22}h_{11}-2g_{12}h_{12}}{2\det g}
            -
            \frac{g_{11}^{(N)}h_{22}^{(N)}+g_{22}^{(N)}h_{11}^{(N)}-2g_{12}^{(N)}h_{12}^{(N)}}{2\det g^{(N)}}
        \right|,
    \end{align}
    where
    \begin{align}
        h_{ij}^{(N)}=\langle\partial_i\partial_jX^{(N)},e^{(N)}\rangle,
        \quad
        e^{(N)}=\frac{\partial_1X^{(N)}\times\partial_2X^{(N)}}{\|\partial_1X^{(N)}\times\partial_2X^{(N)}\|}.
    \end{align}
    By repeatedly applying the very simple equation
    \begin{align}
        ac-bd=\frac12\left[(a+b)(c-d)+(a-b)(c+d)\right],
        \label{eq:simple}
    \end{align}
    the error in the mean curvature can be estimated by evaluating $\|g_{ij}-g_{ij}^{(N)}\|_{L^\infty(B)}$ and $\|h_{ij}-h_{ij}^{(N)}\|_{L^\infty(B)}$.
    Since the estimate of $\|g_{ij}-g_{ij}^{(N)}\|_{L^\infty(B)}$ have already been obtained, we will consider $\|h_{ij}-h_{ij}^{(N)}\|$.
    By applying the product read as inner product in equation \eqref{eq:simple} and the Cauchy--Schwarz inequality, we have
    \begin{align}
        |h_{ij}-h_{ij}^{(N)}|
        &\le
        \frac{1}{2}\left[
            |\langle\partial_i\partial_j X-\partial_i\partial_jX^{(N)},e+e^{(N)}\rangle|
            +
            |\langle\partial_i\partial_jX+\partial_i\partial_jX^{(N)},e-e^{(N)}\rangle|
        \right]\\
        &\le
        \frac12\left[
            |\partial_i\partial_jX-\partial_i\partial_jX^{(N)}||e+e^{(N)}|
            +
            |\partial_i\partial_jX+\partial_i\partial_jX^{(N)}||e-e^{(N)}|
        \right]\\
        &\le
        \frac{3}{2}\left[
            \|\partial_i\partial_jX-\partial_i\partial_jX^{(N)}\|_{L^\infty(B)}\|e+e^{(N)}\|_{L^\infty(B)}
        \right.\\
        &\hspace{100pt}
        \left.
            +
            \|\partial_i\partial_jX+\partial_i\partial_jX^{(N)}\|_{L^\infty(B)}\|e-e^{(N)}\|_{L^\infty(B)},
        \right]
    \end{align}
    where $\|F\|_{L^\infty(B)}$ for an $\mathbb{R}^3$-valued function $F=(F_1,F_2,F_3)$ is defined as
    \begin{align}
        \|F\|_{L^\infty(B)}=\max\{\|F_1\|_{L^\infty(B)},\|F_2\|_{L^\infty(B)},\|F_3\|_{L^\infty(B)}\}.
    \end{align}
    Since $\|\partial_i\partial_jX-\partial_i\partial_jX^{(N)}\|_{L^\infty(B)}$ and $\|e-e^{(N)}\|_{L^\infty(B)}$ converge to $0$ as $N\to\infty$, we see that
    \begin{align}
        \|h_{ij}-h_{ij}^{(N)}\|_{L^\infty(B)}\le C\varepsilon
    \end{align}
    for sufficiently large $N$.
    Combining these estimates, we conclude that 
    \begin{align}
        \|H_{X^{(N)}}\|_{L^\infty(B)}<C\varepsilon
    \end{align}
    holds for sufficiently large $N$.
\end{proof}
\color{black}

\section{Numerical examples}
\label{sec:numerics}

In this section, we show several results of numerical experiments, which exemplify the effectiveness of our method.
We briefly explain how we obtain an $\varepsilon$-conformal configuration $\phi^{(N)}$.
In what follows, for several rectifiable closed Jordan curves $b:\theta\in \partial B\to \Gamma\subset \re^3$, we perform numerical experiments to describe the behavior of the proposed optimization method.
All were carried out by using Julia 1.8.0 on a machine with 3.2 GHz Apple M1 Ultra 20 cores, 128 GB memory, OS X 12.5.1. 
In every computation, the Nesterov iteration is performed $10^5$ times.

\subsection{Jordan domain in the plane}
As a first example, consider the case where the curve $\Gamma$ is embedded in the plane. 
In this case, the problem of finding the minimal surface is nothing more than finding an isometric map from the unit disk onto the Jordan domain bounded by $\Gamma$.

We compute conformal mappings for the ellipse $b(\theta)=(2\cos\theta,\sin\theta,0)$ and the Cassini oval $b(\theta)=(r(\theta)\cos\theta,r(\theta)\sin\theta,0)$, where 
\begin{align}
	r(\theta)= \sqrt{\cos2\theta+\sqrt{1.1^4-\sin^22\theta}}.
\end{align}
For both cases, we choose equidistant 150 points as initial data.
It takes 124.28 seconds for the ellipse and 153.29 seconds for the oval to complete the computation. 
As a result, $\varepsilon$-conformal configurations in Fig.~\ref{fig:ellipse} and Fig.~\ref{fig:oval} are obtained. 
The contour plot of the dilatation $|\Phi_{X^{(N)}}(z)|$ in (b) of the figures confirms the accuracy of the computations. 
Since dilatation is now a holomorphic function, it should obey the maximum principle. 
For both figures, we see some focusing nodes around the circle $\{|z|=0.7\}$. 
In the inner side of the focusing nodes, the dilatation is smaller than $10^{-10}$. 

Ideally, the number of the focusing nodes equals $N$ and they are placed on circle $\{|z|=\rho\}$. 
However, since the computation is carried out with $\rho=0.87$, the focusing nodes are formed inside $\{|z|=\rho\}$.
In addition, the numbers of nodes are 64 for the ellipse and 52 for the oval, which is less than $N=150$. 
Hence, the structure of the focusing nodes yields the discrete version of the maximum principle, but it is not the $N$-points on circle $\{|z|=\rho\}$. 
\begin{figure}[htbp]
\begin{center}
\includegraphics[scale=0.3]{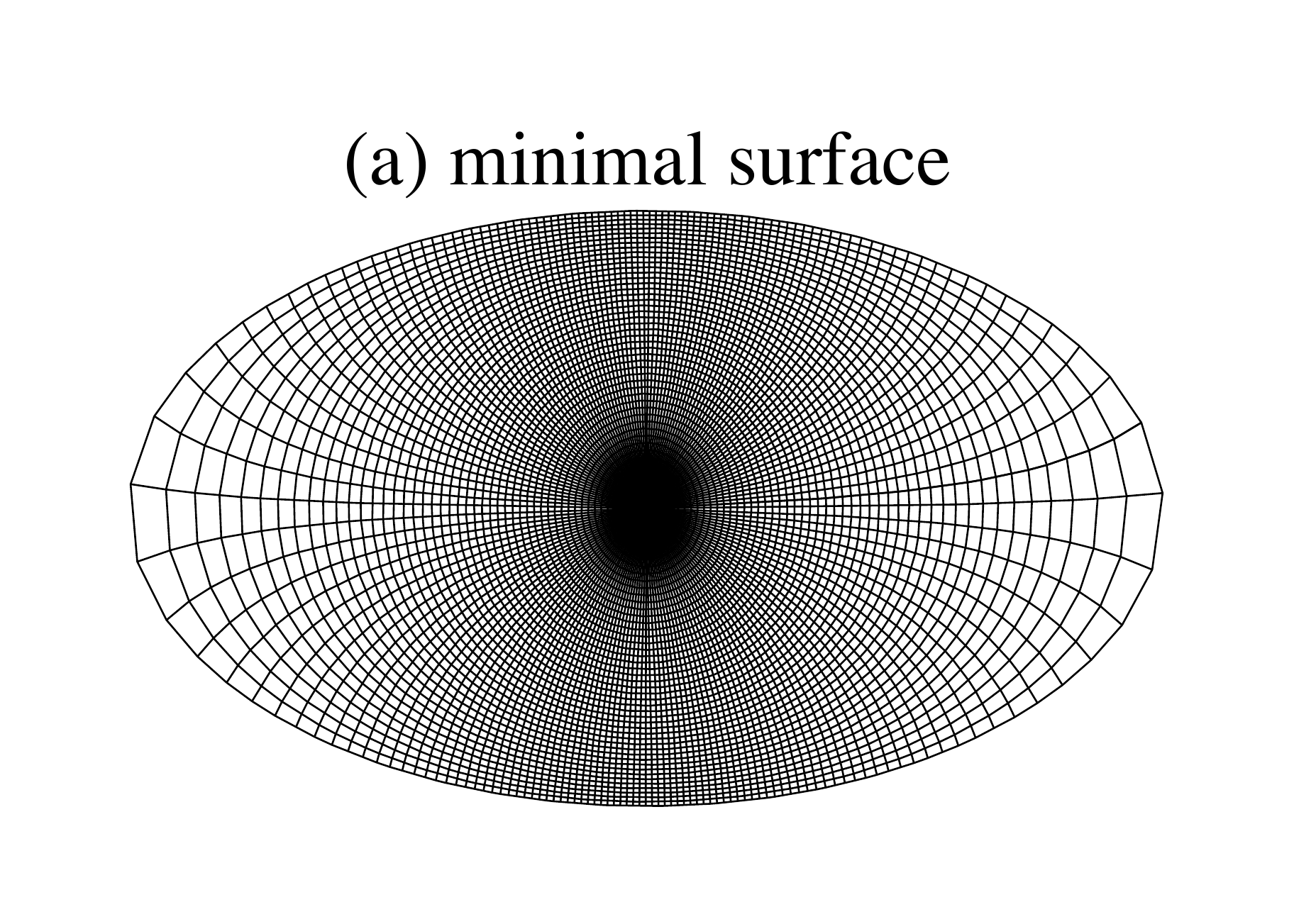}
\includegraphics[scale=0.3]{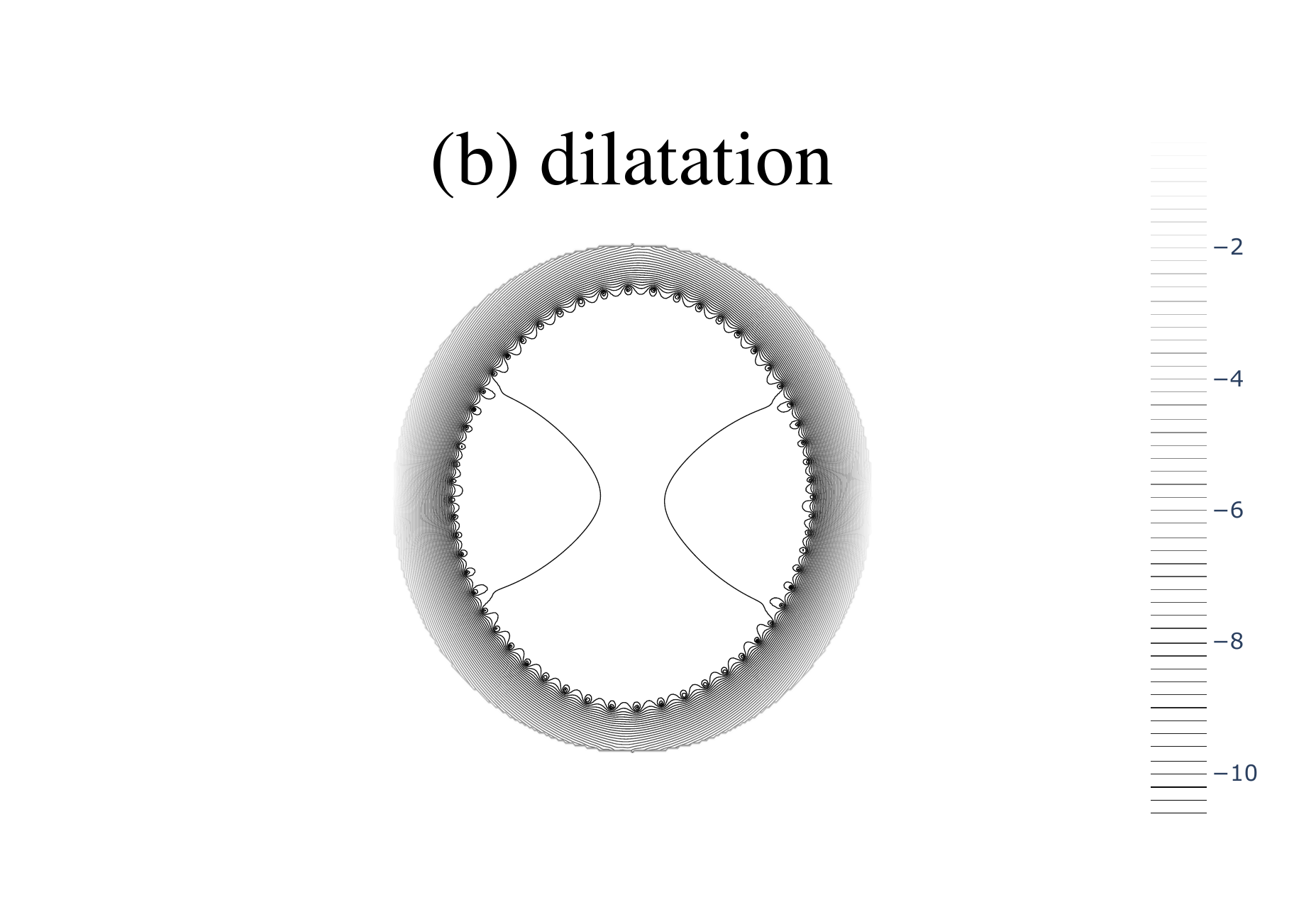}
\includegraphics[scale=0.3]{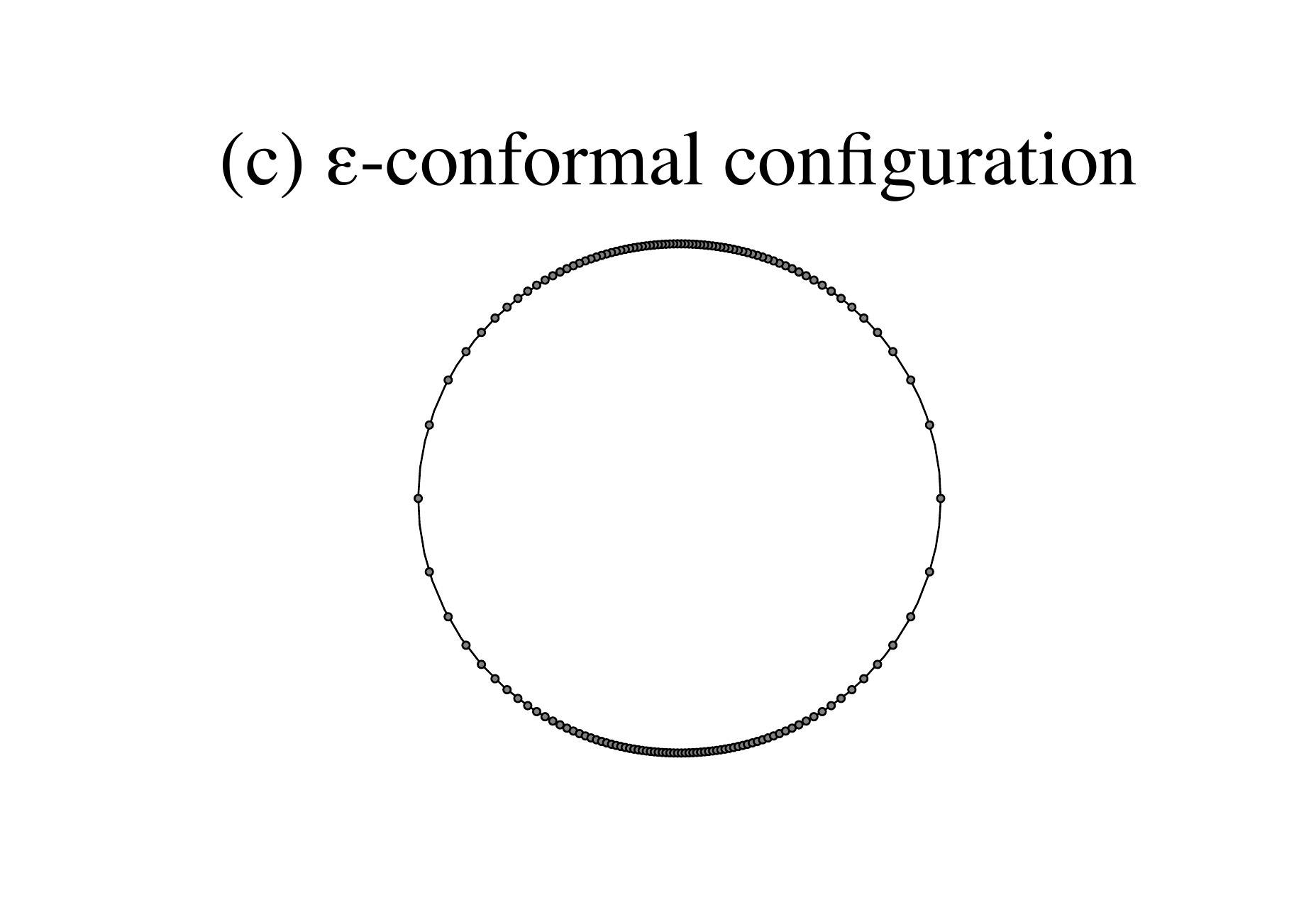}
\end{center}
\caption{Ellipse.}
\label{fig:ellipse}
\end{figure}

\begin{figure}[htbp]
\begin{center}
\includegraphics[scale=0.3]{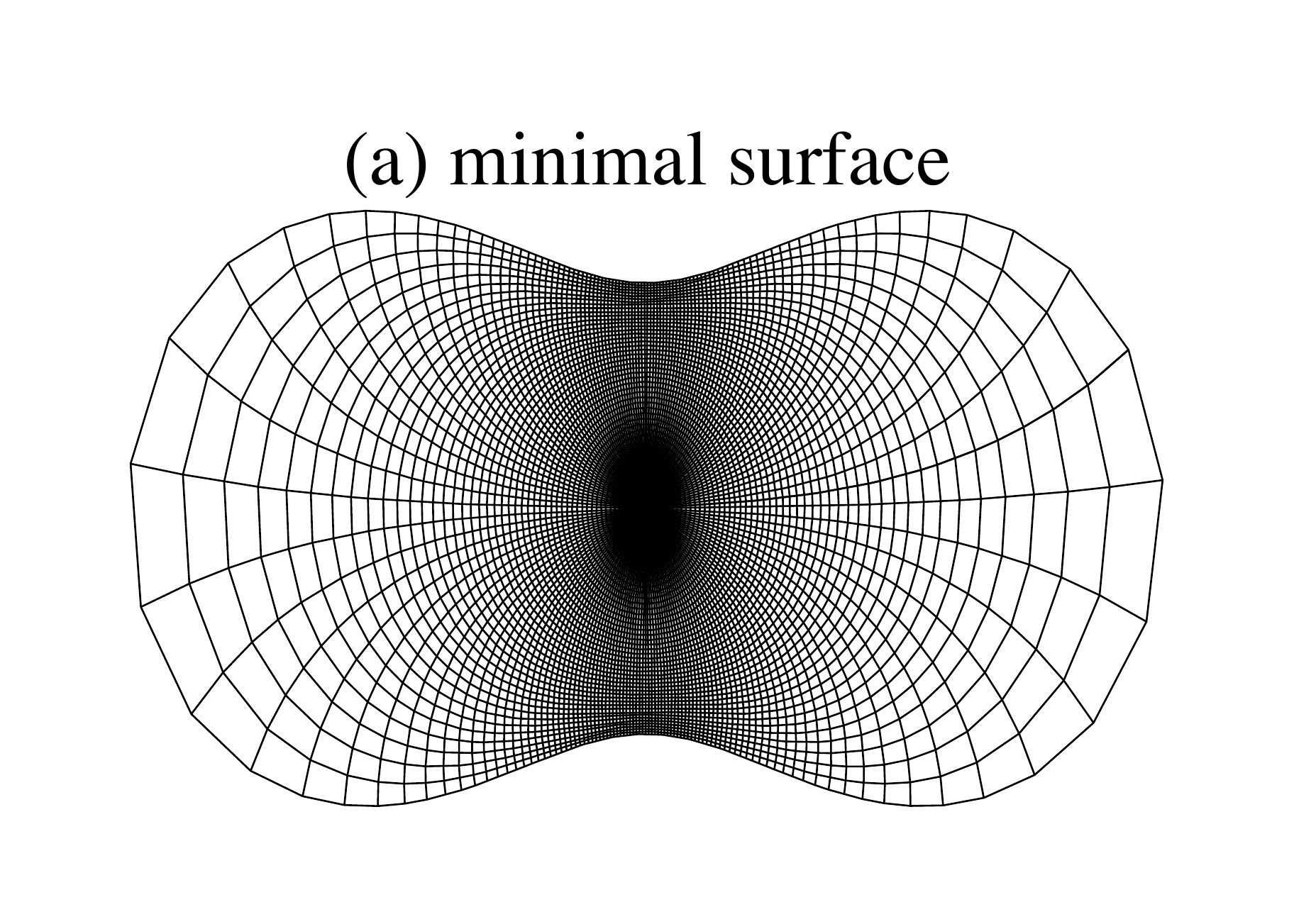}
\includegraphics[scale=0.3]{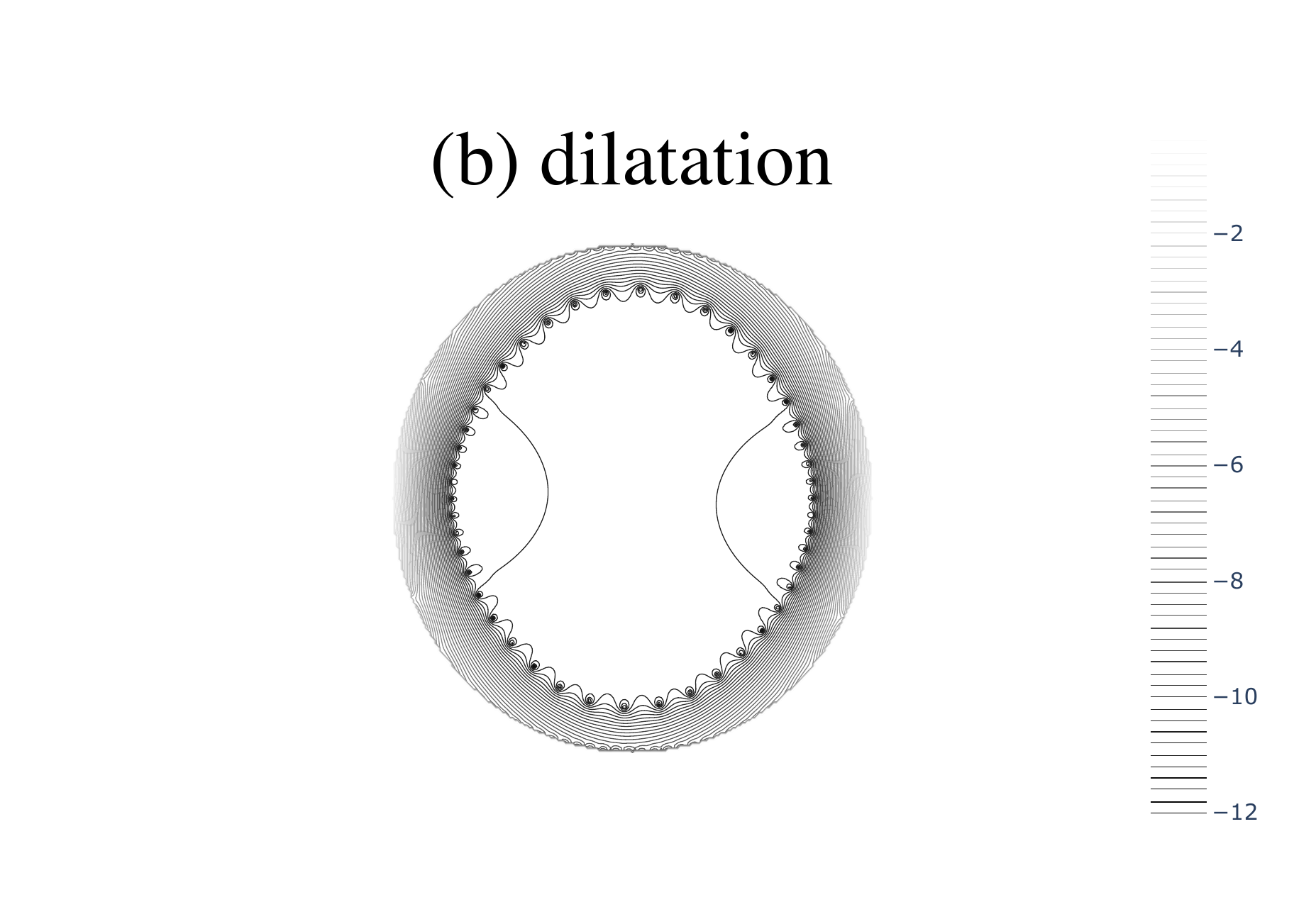}
\includegraphics[scale=0.3]{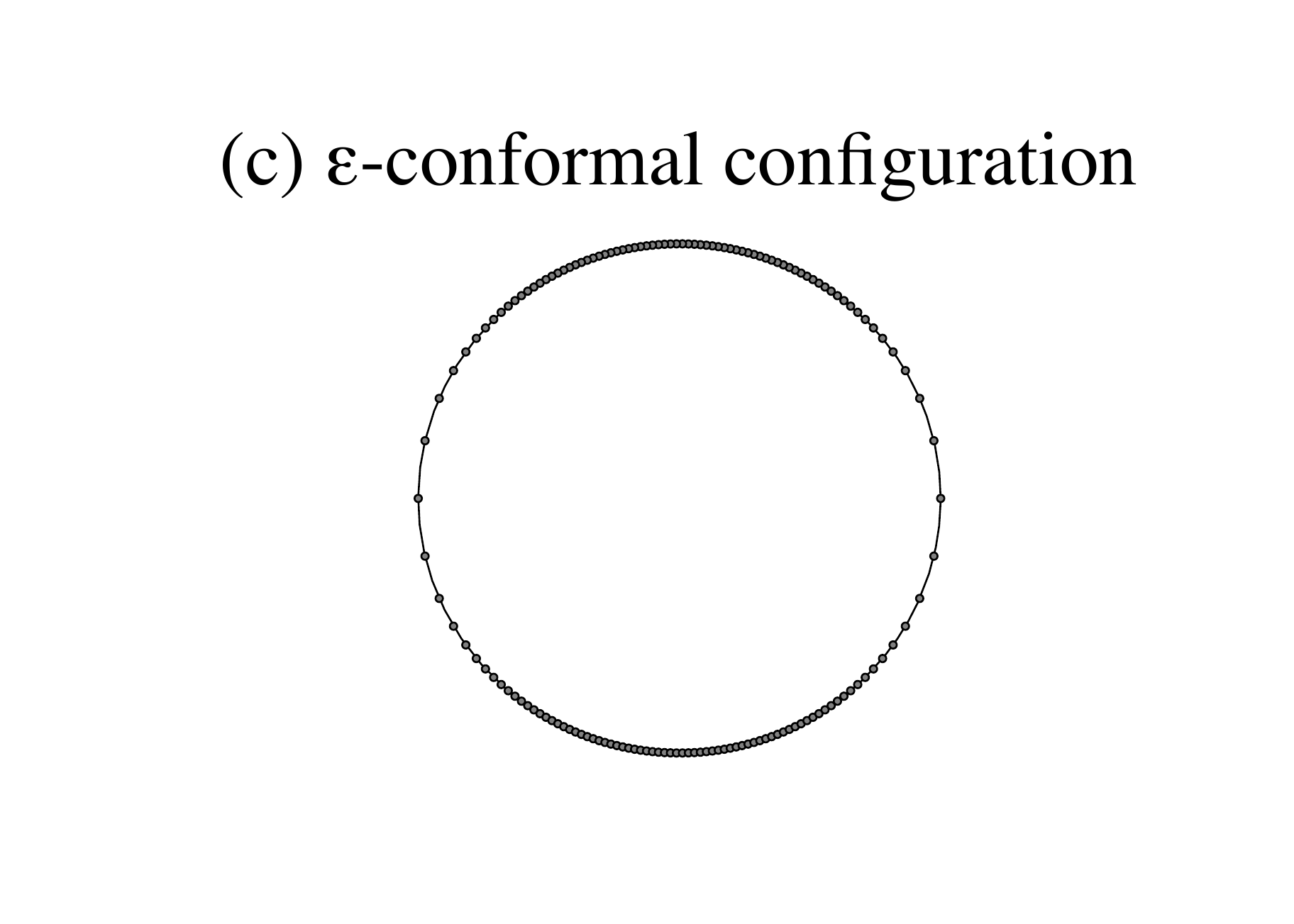}
\end{center}
\caption{Cassini oval.}
\label{fig:oval}
\end{figure}

\subsection{Crown}
Computations not included in the planar domain are as accurate and fast as those included.
A minimal surface spanned by a crown-shaped curve $b(\theta)=(\cos\theta,\sin\theta, 0.3\sin n\theta)$ is demonstrated in Fig.~\ref{fig:crown5} for $n=5$.
The focusing nodes are on circle $\{|z|=0.80\}$ with $\rho=0.9$, and the dilatation inside the focusing nodes is smaller than $10^{-10}$. 
It costs 143.71 seconds in our computation when we choose equidistant 150 points as initial data.
It is worth noting that the focusing nodes form two distinct regular polygons: $8$-polygon inside and $10$-polygon outside. 
If we reduce the number of times of the Nesterov iteration from $10^5$ to $10^4$, we still have the same result, which implies the optimization is so fast and the polygonal structure of the focusing nodes is an intrinsic nature of the crown-shaped curve. 
\begin{figure}[htbp]
\begin{center}
\includegraphics[scale=0.3]{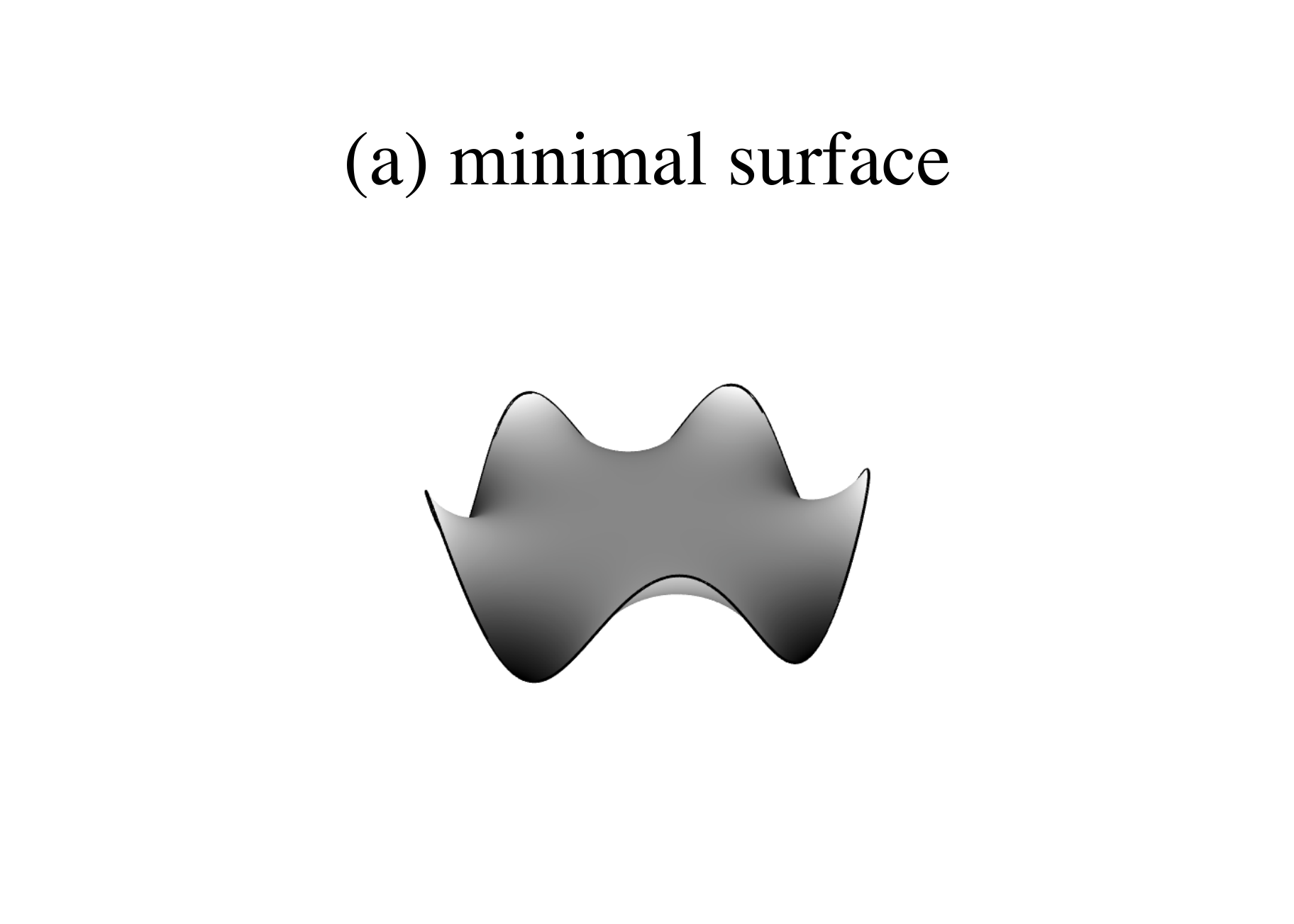}
\includegraphics[scale=0.3]{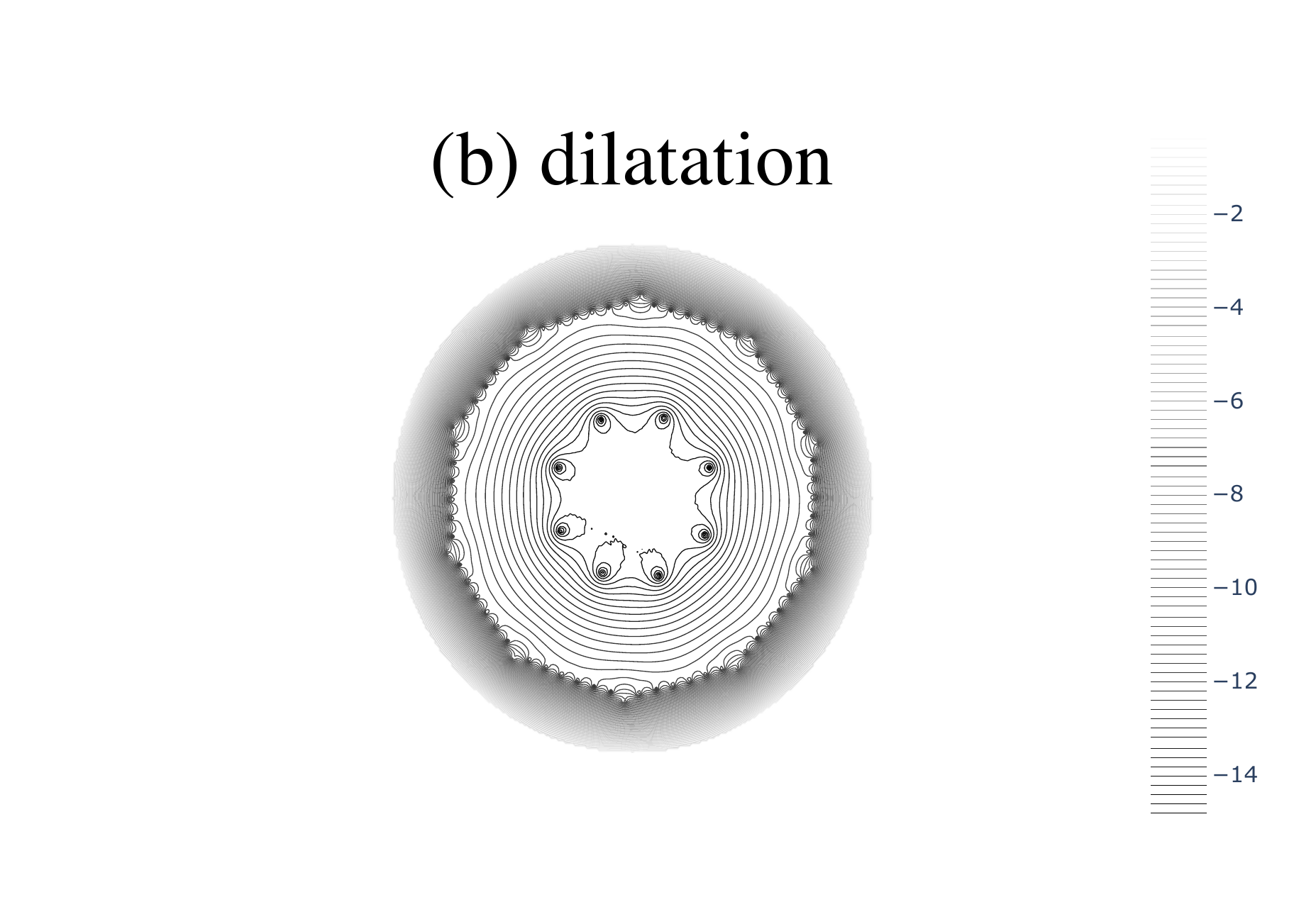}
\includegraphics[scale=0.3]{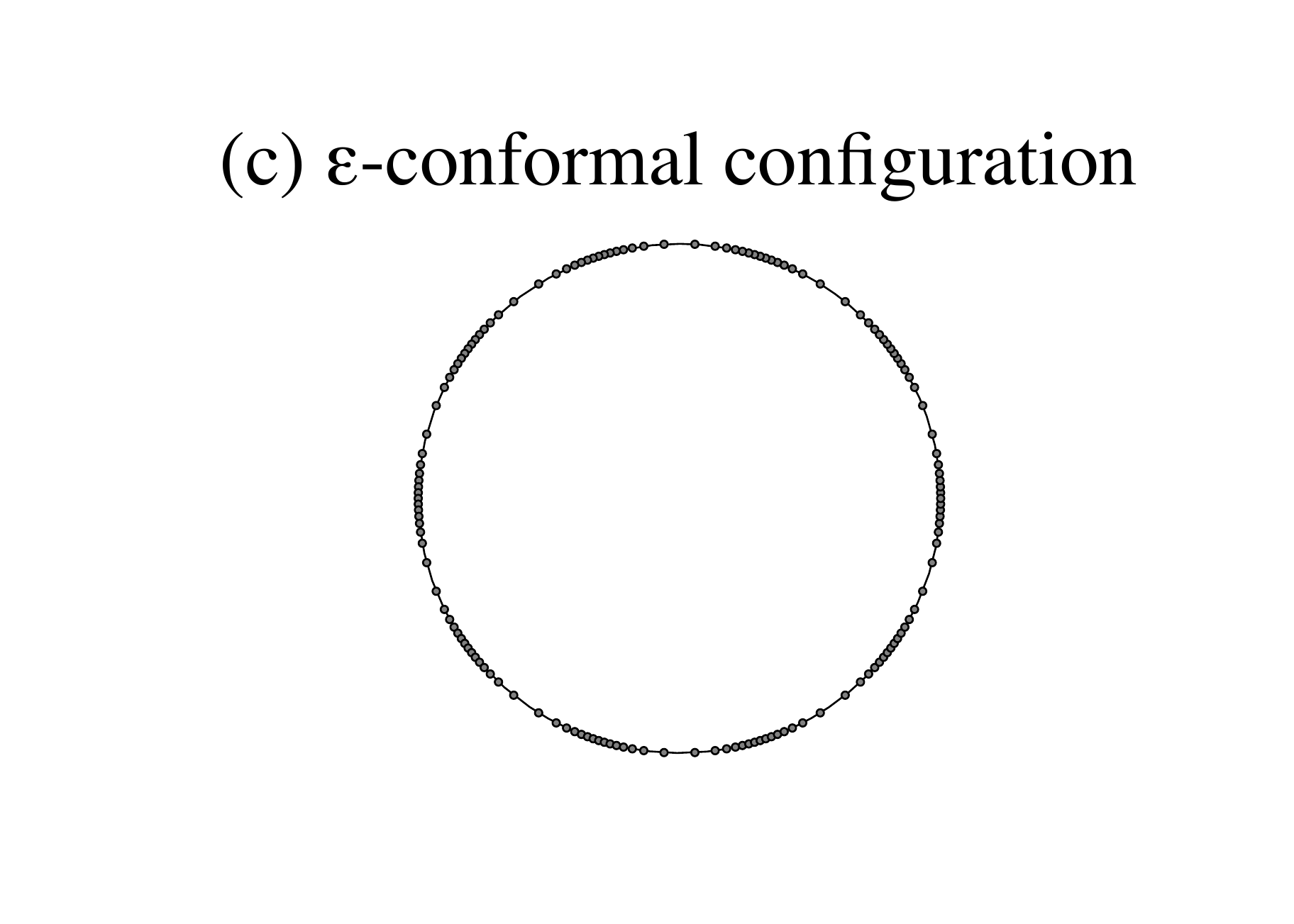}
\end{center}
\caption{Crown for $n=5$.}
\label{fig:crown5}
\end{figure}
 
\subsection{Torus knot}
Even if the resulting minimal surfaces can not be embedded but immersed, and even if a given boundary curve is knotting, we can compute them with the same accuracy and speed. 
The torus knot $b(\theta)=((2+\cos q\theta)\cos p\theta,(2+\cos q\theta)\sin p\theta, -\sin q\theta)$ is a knot and gives such a minimal surface. 
Fig.~\ref{fig:knot32} shows the result for $(3,2)$-torus knot.
Taking equidistant 150 points as initial data, we obtain the result in Fig.~\ref{fig:knot32} in 156.88 seconds for $(3,2)$-torus knot with $\rho=0.85$. 
The focusing nodes are placed on $\{|z|=0.70\}$ with $10^{-10}$ accuracy. 

\begin{figure}[htbp]
\begin{center}
\includegraphics[scale=0.3]{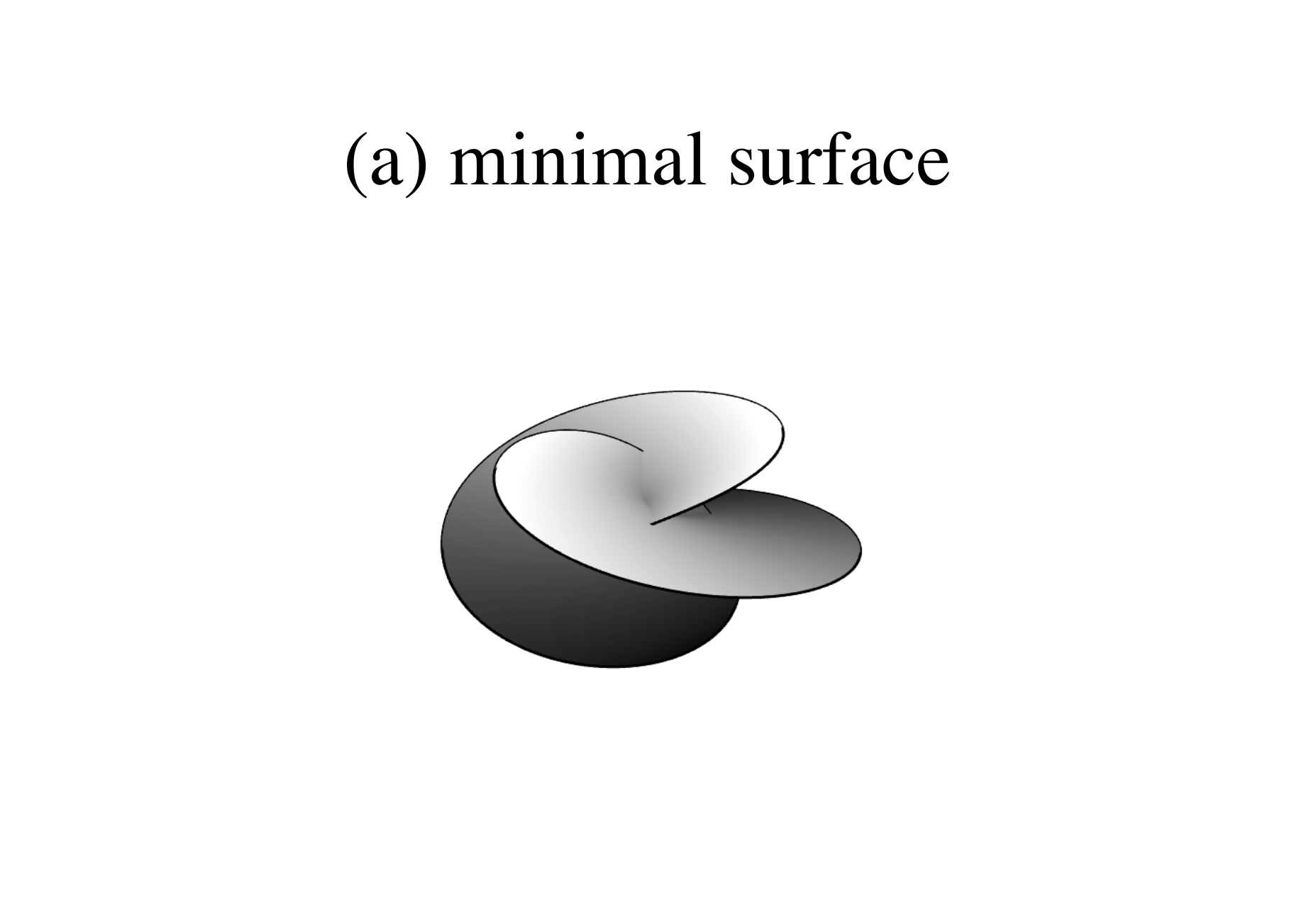}
\includegraphics[scale=0.3]{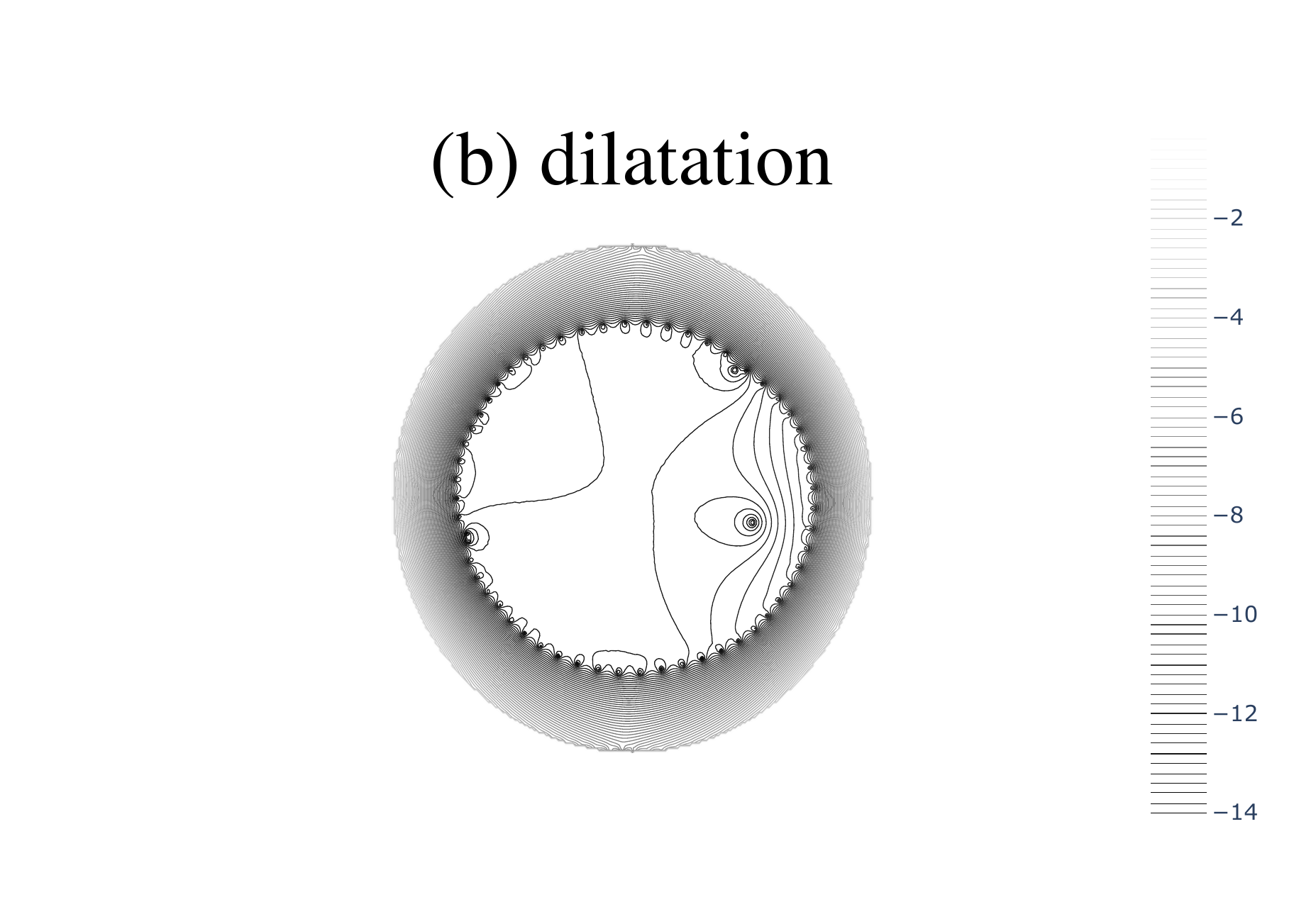}
\includegraphics[scale=0.3]{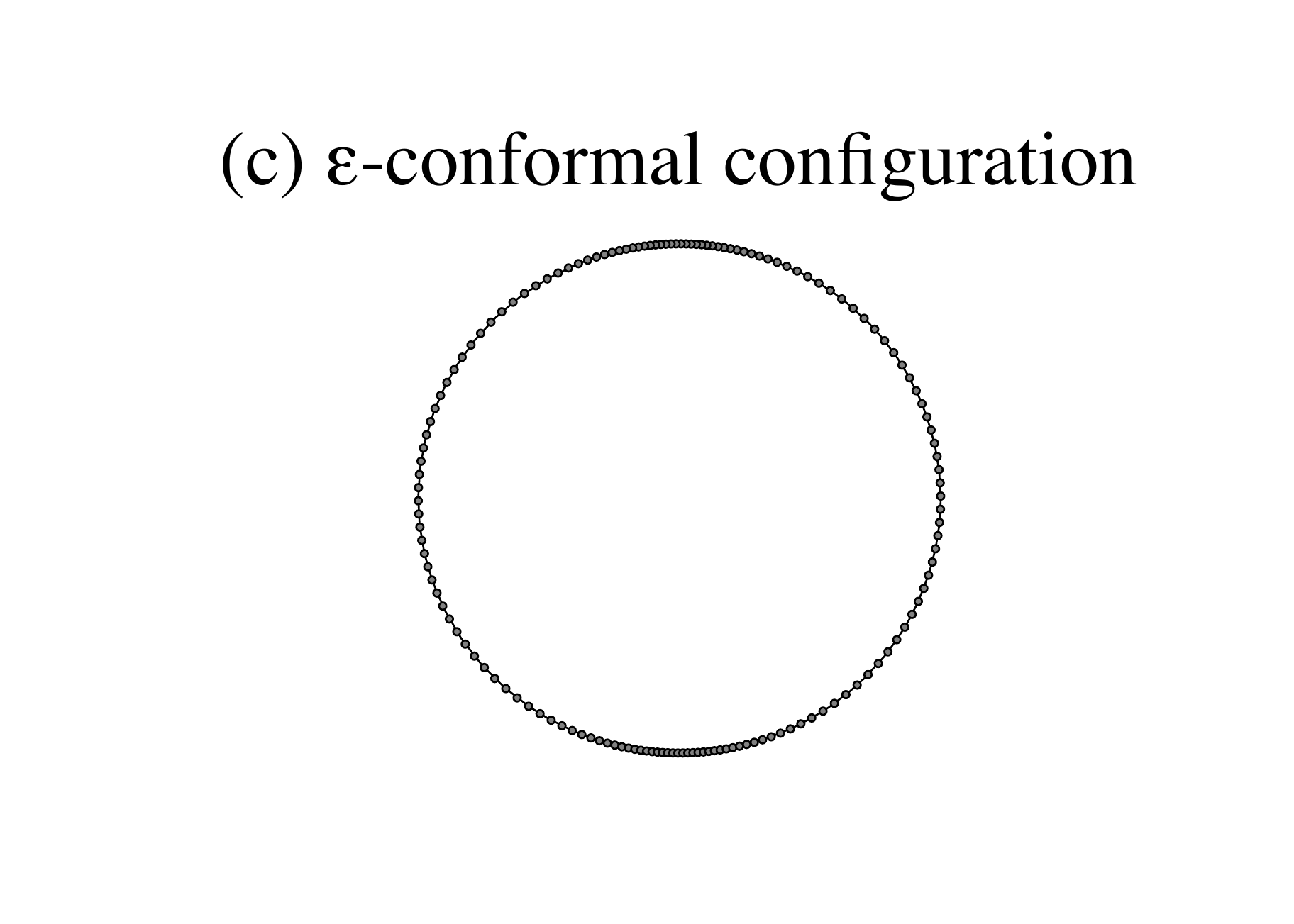}
\end{center}
\caption{Knot for $(p,q)=(3,2)$.}
\label{fig:knot32}
\end{figure}

\section{Searching methods for all solutions}
\label{sec:search_all}

We finally propose two methods to search for all solutions to the Plateau problem for a given boundary curve. 
We demonstrate these methods by taking the Enneper wire as an example, and the Enneper wire is a closed Jordan curve given by 
\begin{align}
	b_1(\theta) &= r\cos\theta-\frac{r^{3}}{3}\cos 3\theta,\\
	b_2(\theta) &= -r\sin\theta-\frac{r^{3}}{3}\sin 3\theta,\\
	b_3(\theta) &= r^{2}\cos 2\theta.
\end{align}
for $r\in(0,\sqrt{3})$. 
The Enneper wire comes from the boundary of the Enneper surface, which is an exact solution to the Plateau problem. 
The homeomorphism $\phi:\partial B\to \partial B$ is then the identity map on $\partial B$. 
Hence, the equidistant point set is the $\varepsilon$-conformal configuration corresponding to this exact solution.

Fig.~\ref{fig:exact_enneper} shows the numerical solution for the equidistant points corresponding to the Enneper surface. 
We see no focusing nodes in the contour plot of the dilatation, which implies that focusing nodes cannot be formed without optimization.
It is known that the Enneper wire bounds two distinct minimal surfaces other than the Enneper surface. 

Our goals in this section are to find solutions and to confirm that there is no other solution than what we have found.
First, let us search for solutions in a unified manner as follows. 
We take a one-parameter family of initial configurations $\bm{\phi}^{(N)}: s\in (S_1,S_2)\to \boldsymbol{\phi}^{(N)}(s)\in\mathbb{T}^N$ with 
\begin{align}
	\phi_j^{(N)}(s) = \frac{2\pi(j-1)}{N}+s \sin\frac{2\pi m(j-1)}{N}
\end{align}
for $s>0$ and $m\in\ze$, and $\phi^{(N)}$ is regarded as an associated configuration of a perturbed identity map on $\partial B$ by a Fourier mode $ s\sin mx$. 
Fig.~\ref{fig:all_ini_conf} indicates that each initial configurations converges to either of the three minimizers for $s\in\{-2.95,-2.9,\ldots,2\}$ and $m=2$. 

Since $\boldsymbol{\phi}^{(N)}(s)$ varies continuously for $s$, the value of the Dirichlet energy also varies continuously.
Hence, we can distinguish the minimizers from connected components of the value of the Dirichlet energy as $s$ varies. 
Fig.~\ref{fig:dirichlet} shows the value of Dirichlet energy for each $s$. 
The color of the points is painted according to the accuracy of $E$. 
We see that the distribution of the values is well-organized if $E$ is more accurate than 5-digit precision; otherwise, it is scattered. 
Therefore, it is suggested that the numerical solutions must be more accurate than at least 5-digit precision to distinguish minimizers. 

We pick out three representatives of the connected components with 5-digit precision, choosing $s=-1,0,1$. 
Fig.~\ref{fig:repr_ini_conf} shows their initial and $\varepsilon$-conformal configurations for $s=-1,0,1$ in left, center, and right, respectively.
Indeed, it is worth noting that these representatives give three distinct minimal surfaces in Fig~\ref{fig:repr_min_surf}. 
They attain 8-digit precision by 160 seconds and form focusing nodes, as in Fig.~\ref{fig:repr_contour}. 
Hence, we can find solutions by taking $\boldsymbol{\phi}^{(N)}(s)$. 

Second, let us choose initial configurations at random, thereby check there is no other solution than what we found above. 
Picking up points randomly and interpolating them by B-spline interpolation, we obtain 5000 samples with 5-digit precision. 
Fig.~\ref{fig:dst} and Fig.~\ref{fig:redst} show the distribution and its rearrangement of the Dirichlet energy according to the accuracy with 5-digit (left), 6-digit (center),  and 7-digit (right) precision.  

We observe that there are at least two solutions by distinguishing these samples by whether the value of the Dirichlet energy is greater than 12 or not. 
When the Dirichlet energy of a sample is greater than 12, the sample corresponds to the case $s=0$. 
Otherwise, they correspond to either the case $s=1$ or $s=-1$. 
Note that the case for $s=1,-1$ cannot be distinguished by the Dirichlet energy since the minimal surfaces for $s=1,-1$ coincide by certain rigid rotation. 
However, we can see from the shape of the minimal surface that there is no solution other than what we have found by taking random initial configurations.

\section{Concluding remarks}
\label{sec:concluding_remarks}
We propose a numerical scheme with such a high speed and high accuracy as to find all minimal surfaces spanned by a closed Jordan curve. 
The numerical scheme is based on the method of fundamental solutions (MFS) and the Nesterov's accelerated gradient decent. 
We can compute an approximate solutions for the Dirichlet boundary value problem by the MFS, and the error of the approximate solution to an exact solution decays exponentially for sufficiently smooth boundary data.
In the computation of a minimal surface, the Dirichlet boundary value problem arises to compute the coordinate function of the minimal surface. 
However, not all surfaces obtained as the solution for the boundary value problem are minimal surfaces because the solution varies as a change of variable of the boundary that gives a different boundary data.
Hence, it is necessary to compute minimal surfaces via the MFS to find a suitable change of variable that yields minimal surfaces. 
We proposed a minimization problem for discrete energy of the complex dilatation around the boundary. 
The significant characteristic is that every MFS approximate solution is smooth and harmonic on the whole domain, eliminating the need for the integration on the whole space of the functional. 
Nesterov's accelerated gradient decent solves this minimization problem. 

For the proposed numerical method, we proved the existence of a solution for the minimization problem for the complex dilatation in Theorem~\ref{thm:existence_approx_surf}. 
We obtained the error estimate of the complex dilatation according to  the regularity of a given boundary. 
We next showed the existence of a subsequence of a given sequence of approximate surfaces and the limit. 
In particular, the limit became a minimal surface and the value of the Dirichlet energy convergeed to the one for the minimal surface. 
We also obtained the $L^\infty$-error estimate for the mean curvature. 

We investigated the performance of numerical computation for the proposed minimal surface. 
We saw from the computation for some boundaries that the complex dilatation satisfied a discrete version of the maximum principle. 
The errors of the complex dilatation were with 10-digit precision for at most 160 seconds for each of boundaries in our computation. 
Lastly, we proposed two methods of finding all minimal surfaces spanned by a given closed Jordan curve. 
We chose a one-parameter family of initial configurations in the first method based on a perturbed identity map on the boundary by a Fourier mode. 
Since the initial configurations changed continuously, we could find distinct minimal surfaces as a minimizer of boundary mapping or connected components of the value of the Dirichlet energy. 
In the second method, we selected the initial configuration at random. 
We can see from the distribution of the Dirichlet energy or the surface shape that there is no solution other than what we found. 
It remains for future work to determine whether the above methods cannot find some minimal surfaces or how many samples must be calculated to ensure that all minimal surfaces have been obtained.

\subsection*{Acknowledgements}

This work was supported by JSPS KAKENHI Grant Numbers JP18K13455, JP22K03425, JP21J00025.

\begin{figure}[htbp]
\begin{center}
\includegraphics[scale=0.3]{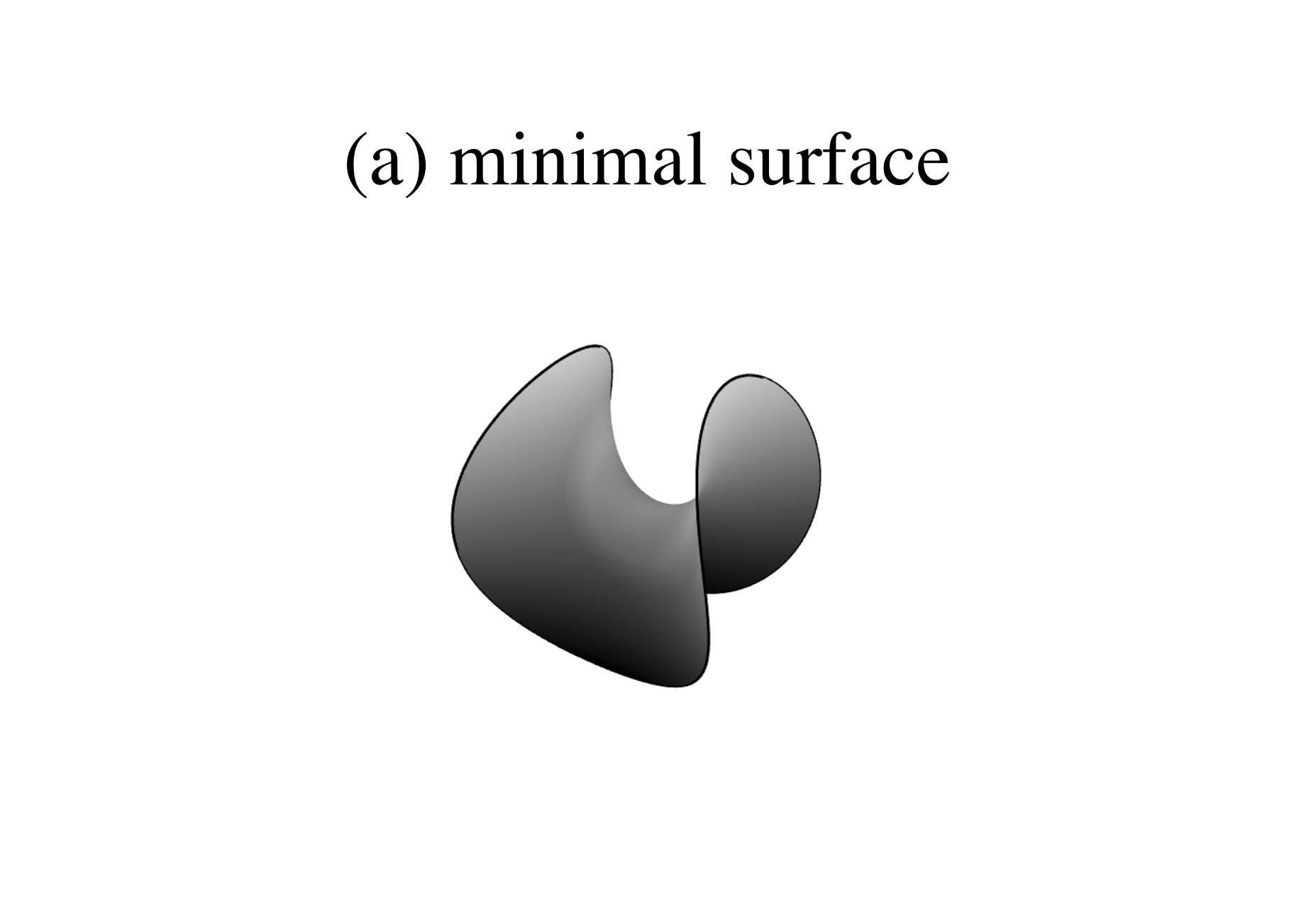}
\includegraphics[scale=0.3]{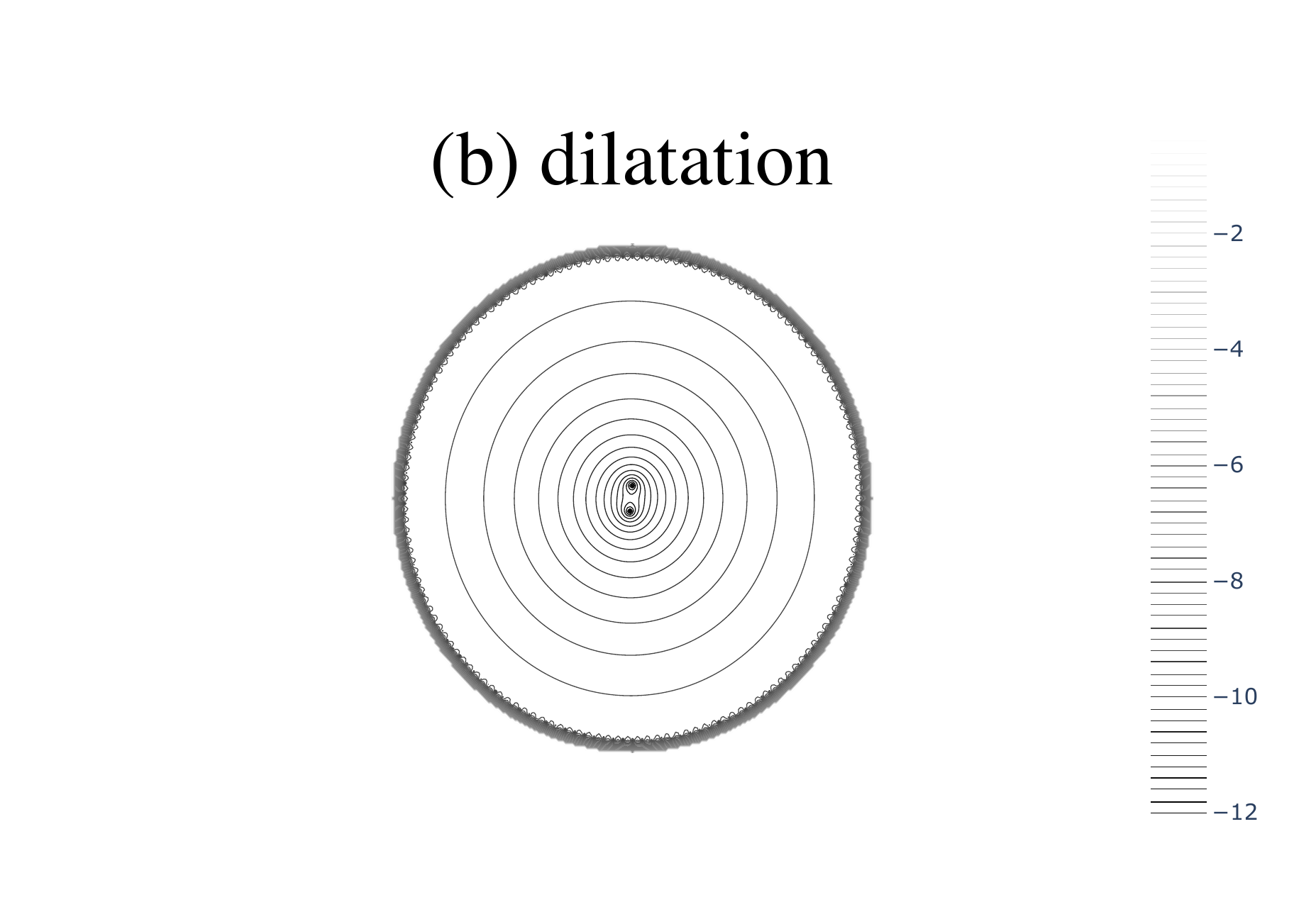}
\includegraphics[scale=0.3]{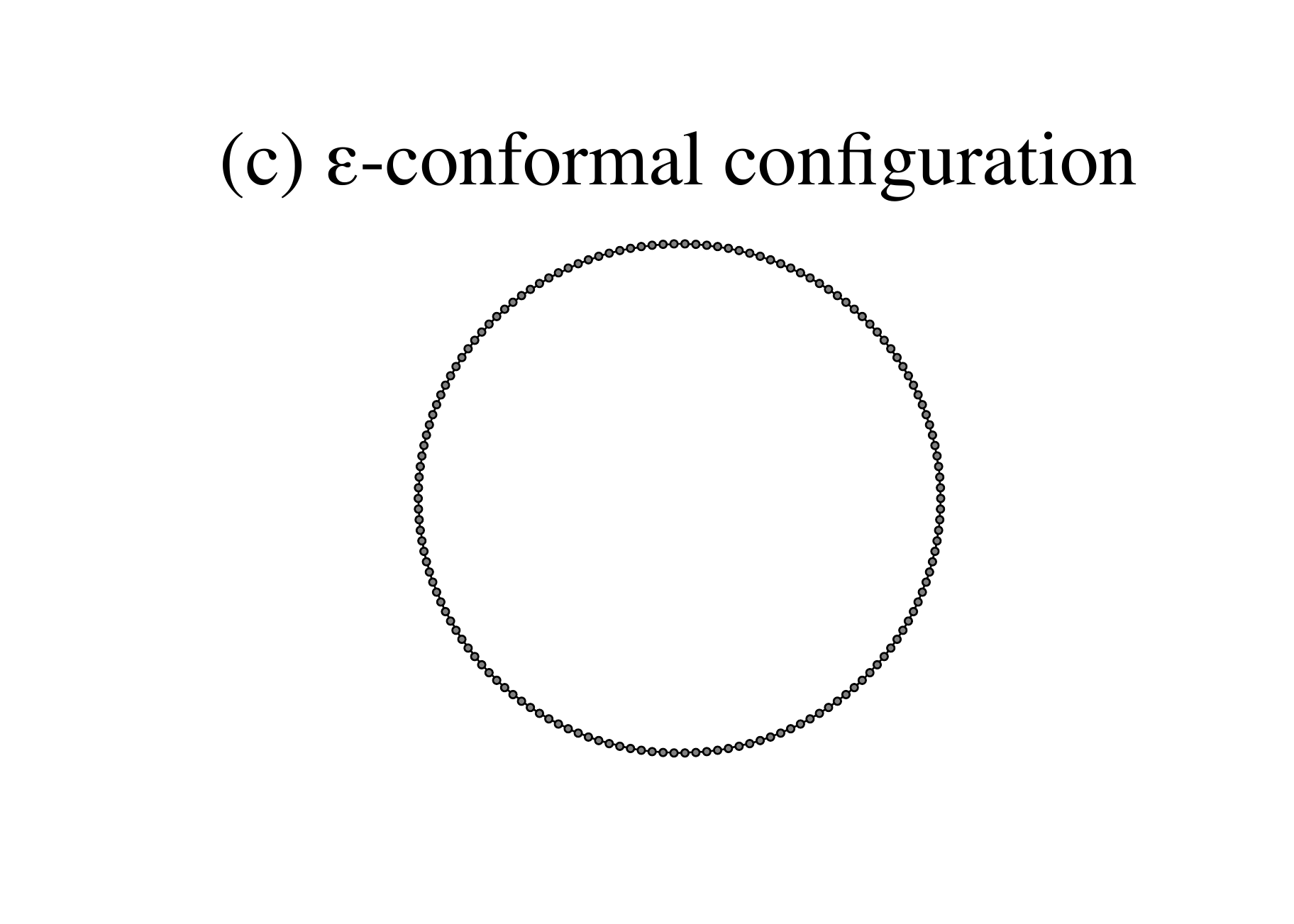}
\end{center}
\caption{Enneper surface for $R=1.1$.}
\label{fig:exact_enneper}
\end{figure}

\begin{figure}[htbp]
\begin{center}
\includegraphics[scale=0.47]{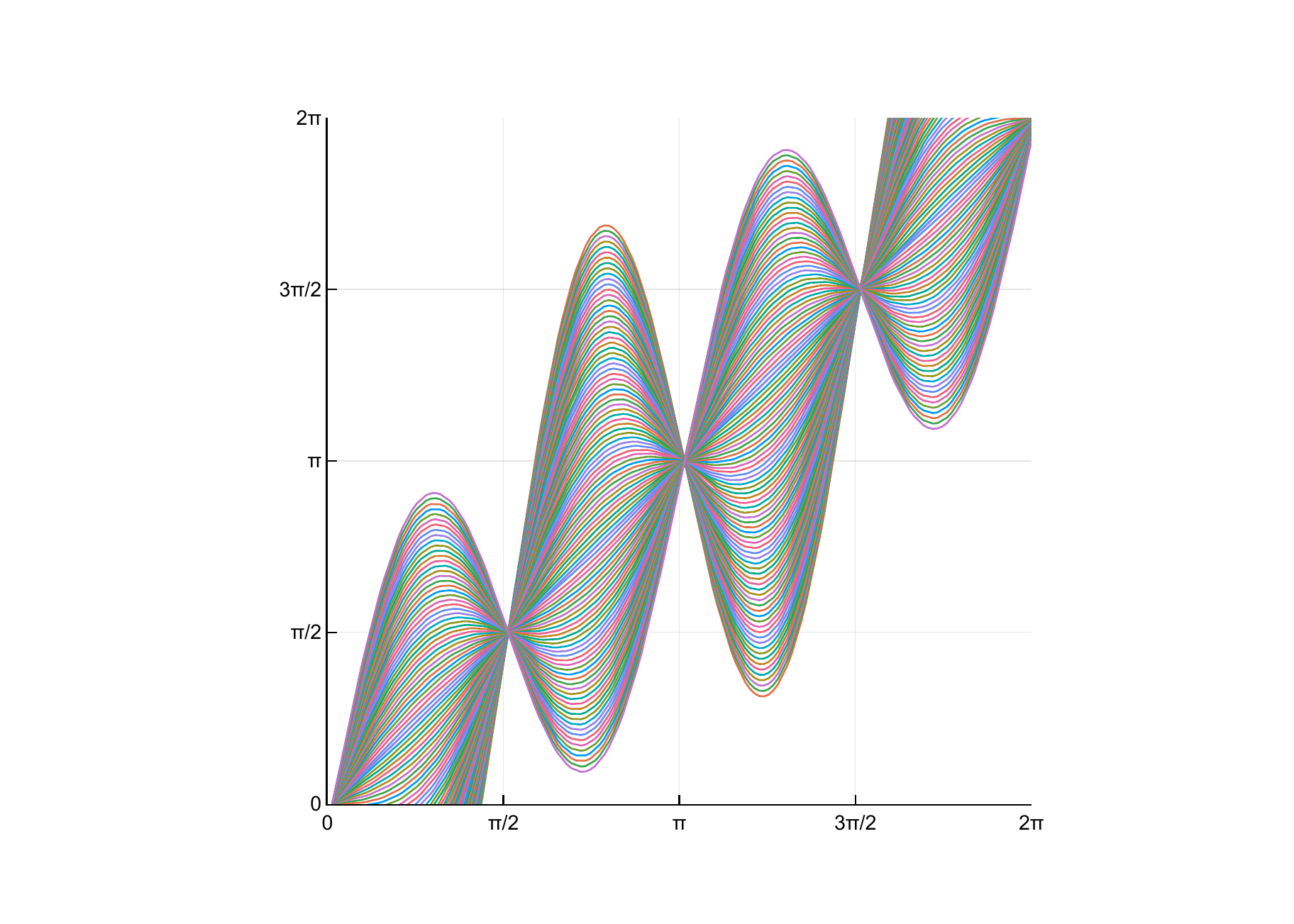}
\includegraphics[scale=0.47]{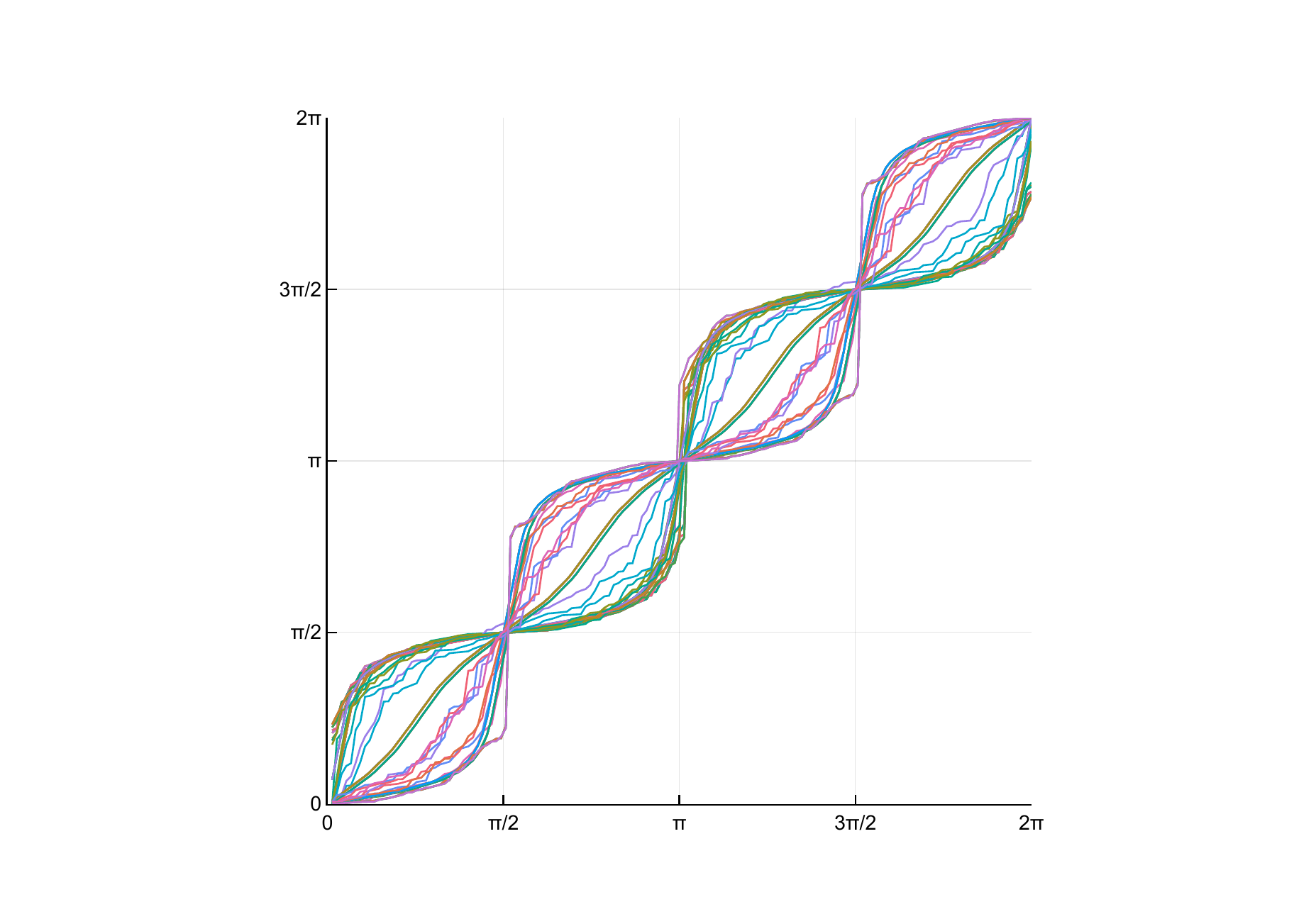}
\end{center}
\caption{$\boldsymbol{\phi}^{(N)}(s)$ (left) and $\varepsilon$-conformal configuration (right).}
\label{fig:all_ini_conf}
\end{figure}

\begin{figure}[htbp]
\begin{center}
\includegraphics[scale=0.5]{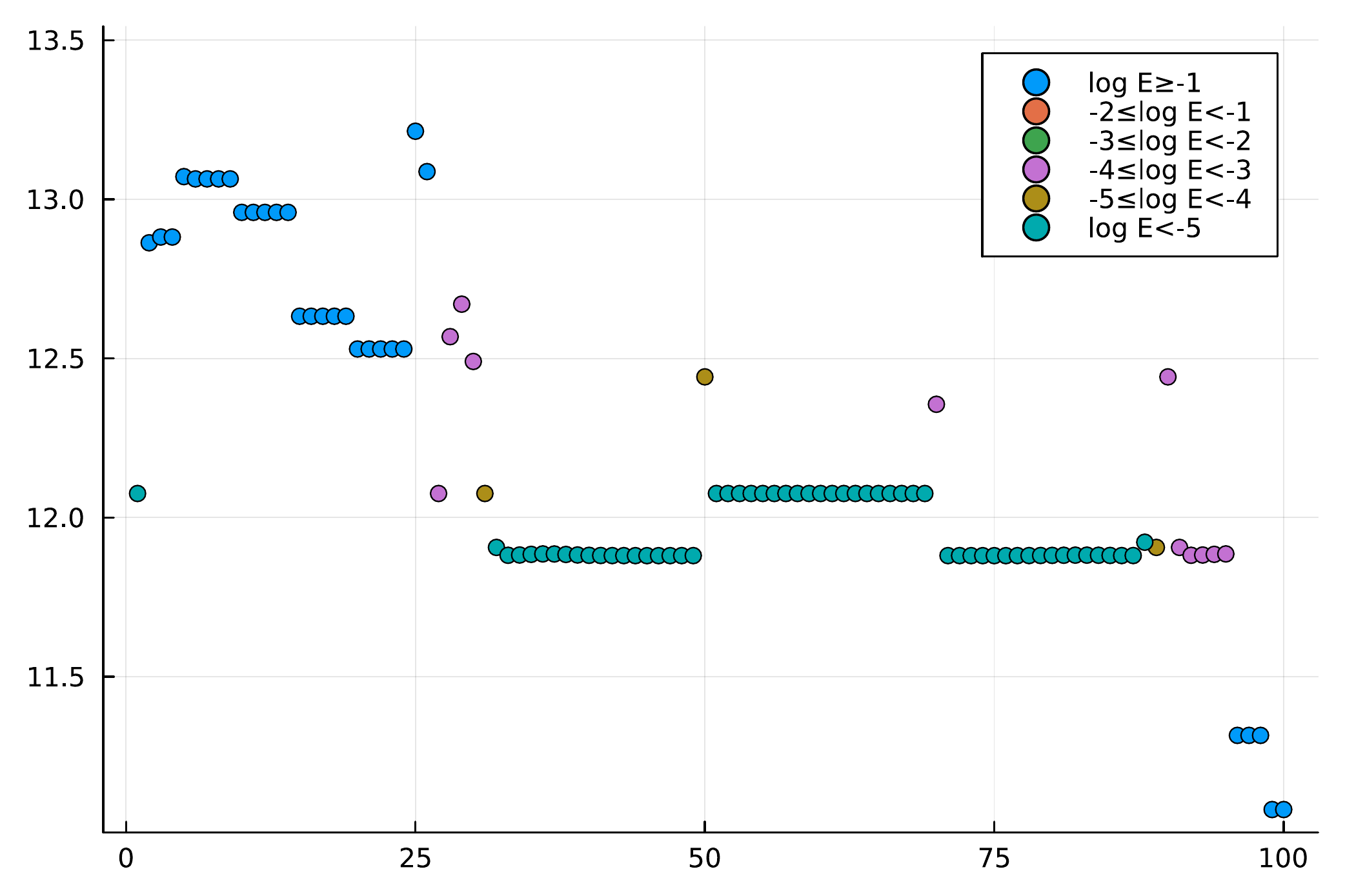}
\end{center}
\caption{Dirichlet energy.}
\label{fig:dirichlet}
\end{figure}

\begin{figure}[htbp]
\begin{center}
\includegraphics[scale=0.3]{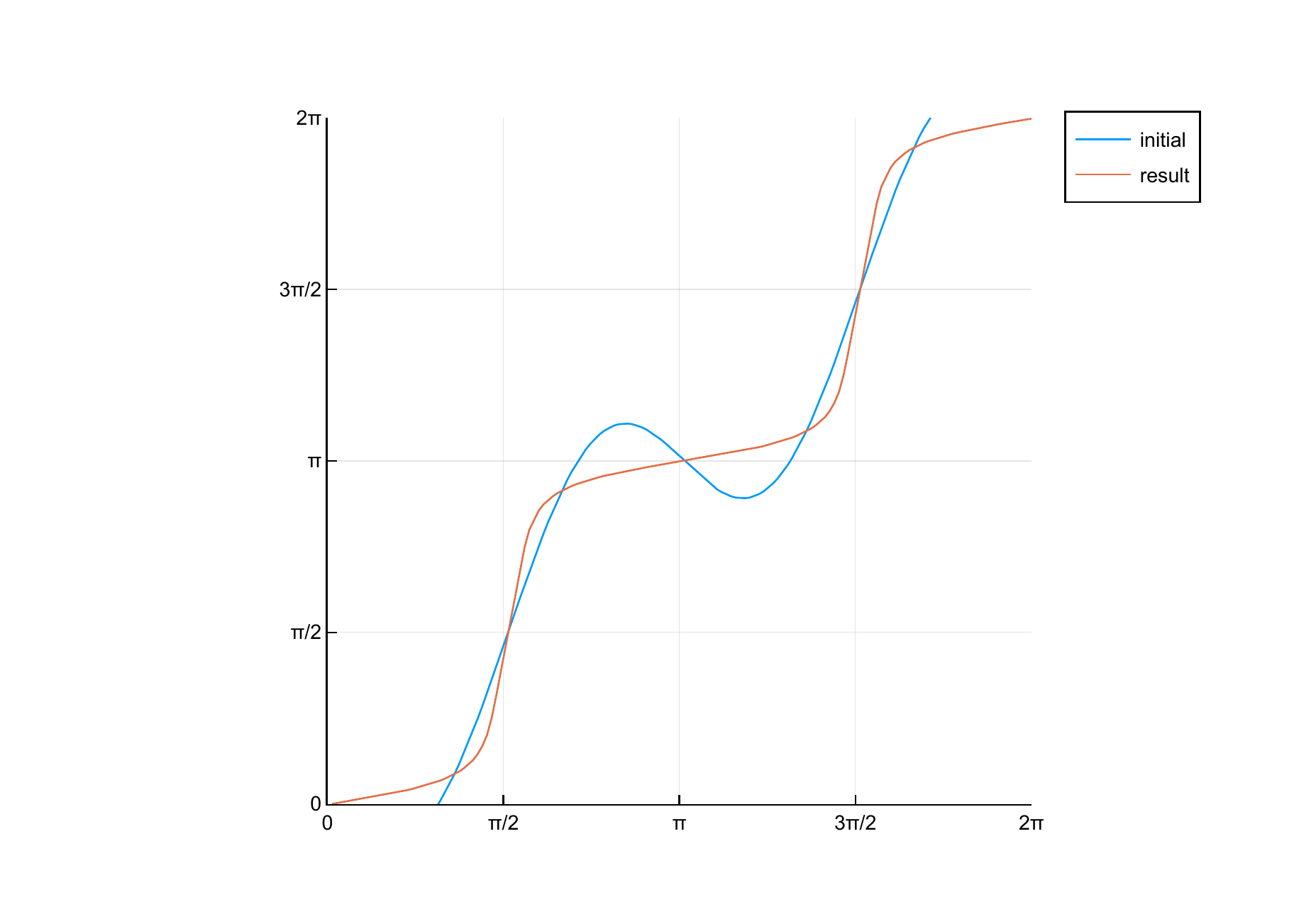}
\includegraphics[scale=0.3]{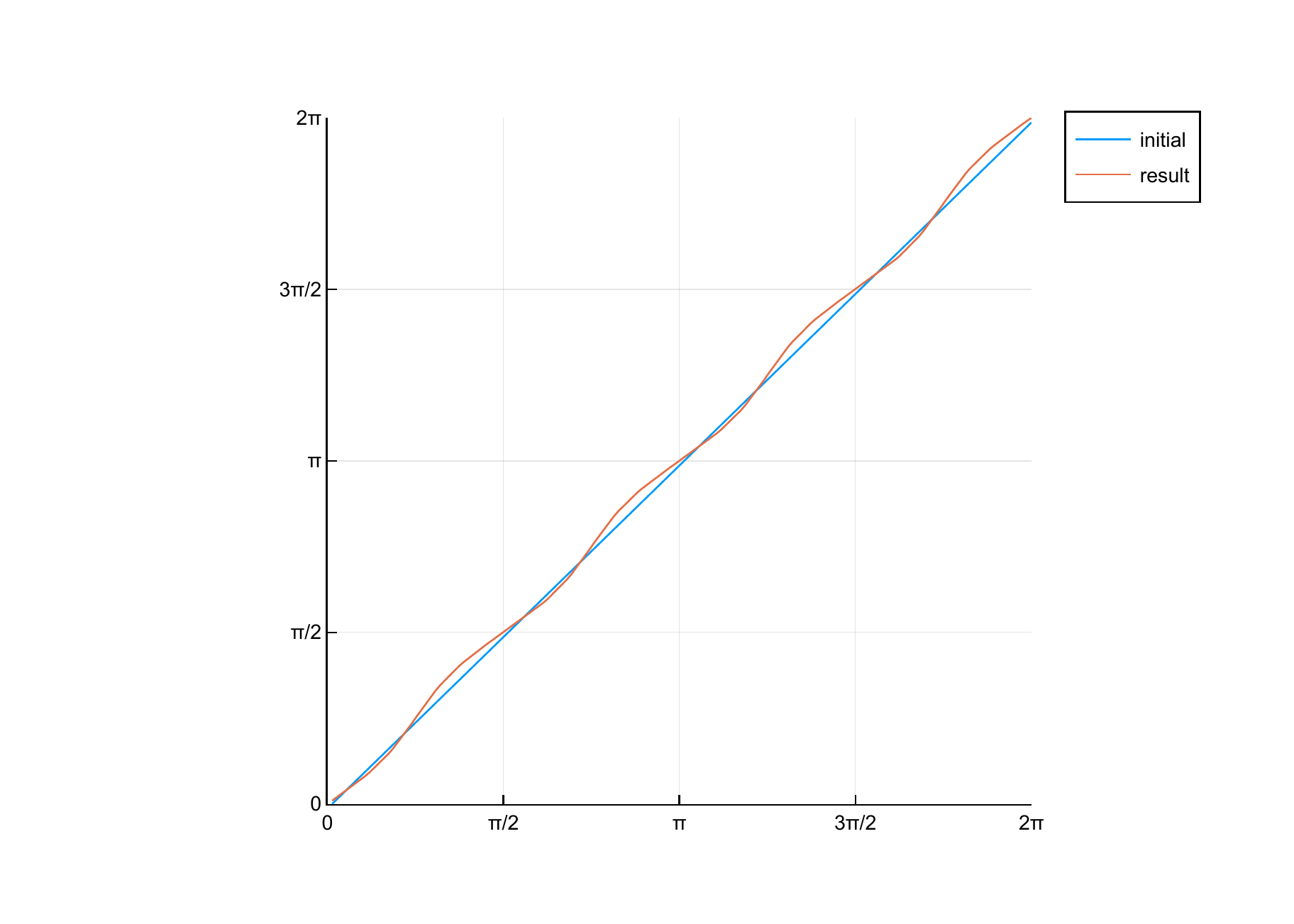}
\includegraphics[scale=0.3]{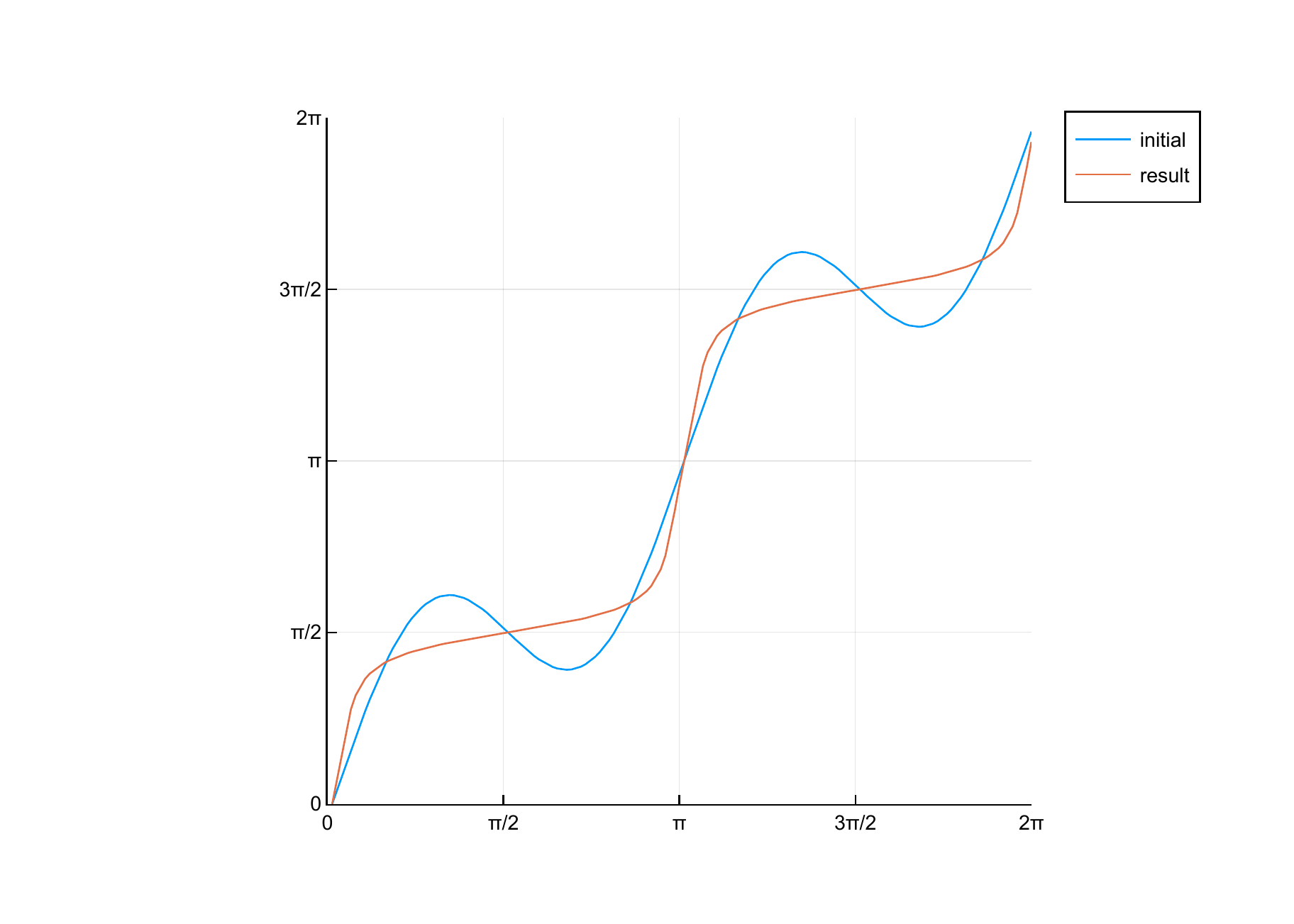}
\end{center}
\caption{Initial and $\varepsilon$-conformal configurations.}
\label{fig:repr_ini_conf}
\end{figure}

\begin{figure}[htbp]
\begin{center}
\includegraphics[scale=0.7]{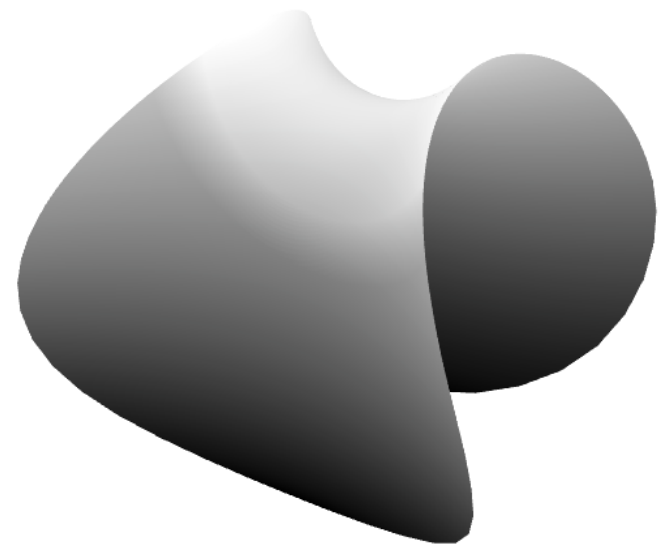}
\includegraphics[scale=0.7]{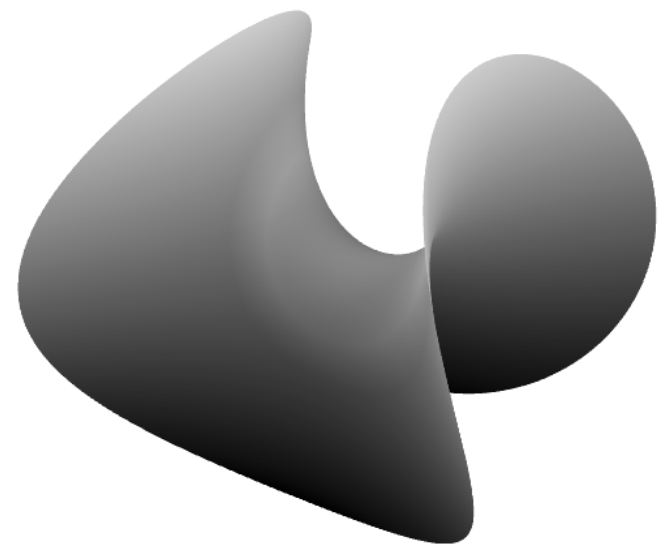}
\includegraphics[scale=0.7]{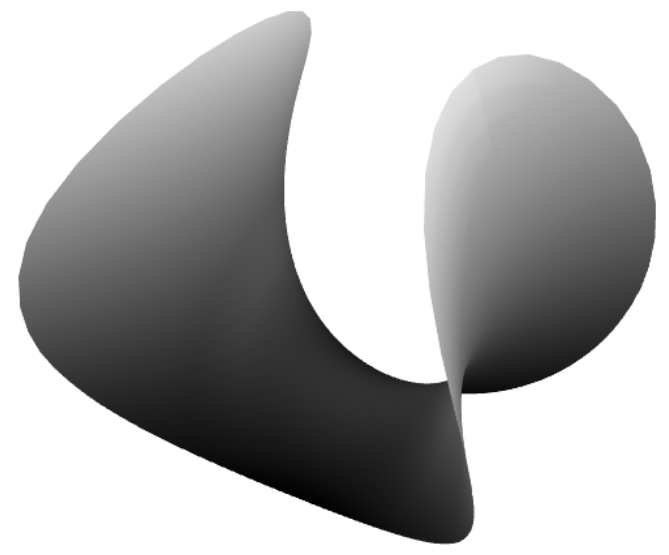}
\end{center}
\caption{Minimal surfaces.}
\label{fig:repr_min_surf}
\end{figure}

\begin{figure}[htbp]
\begin{center}
\includegraphics[scale=0.3]{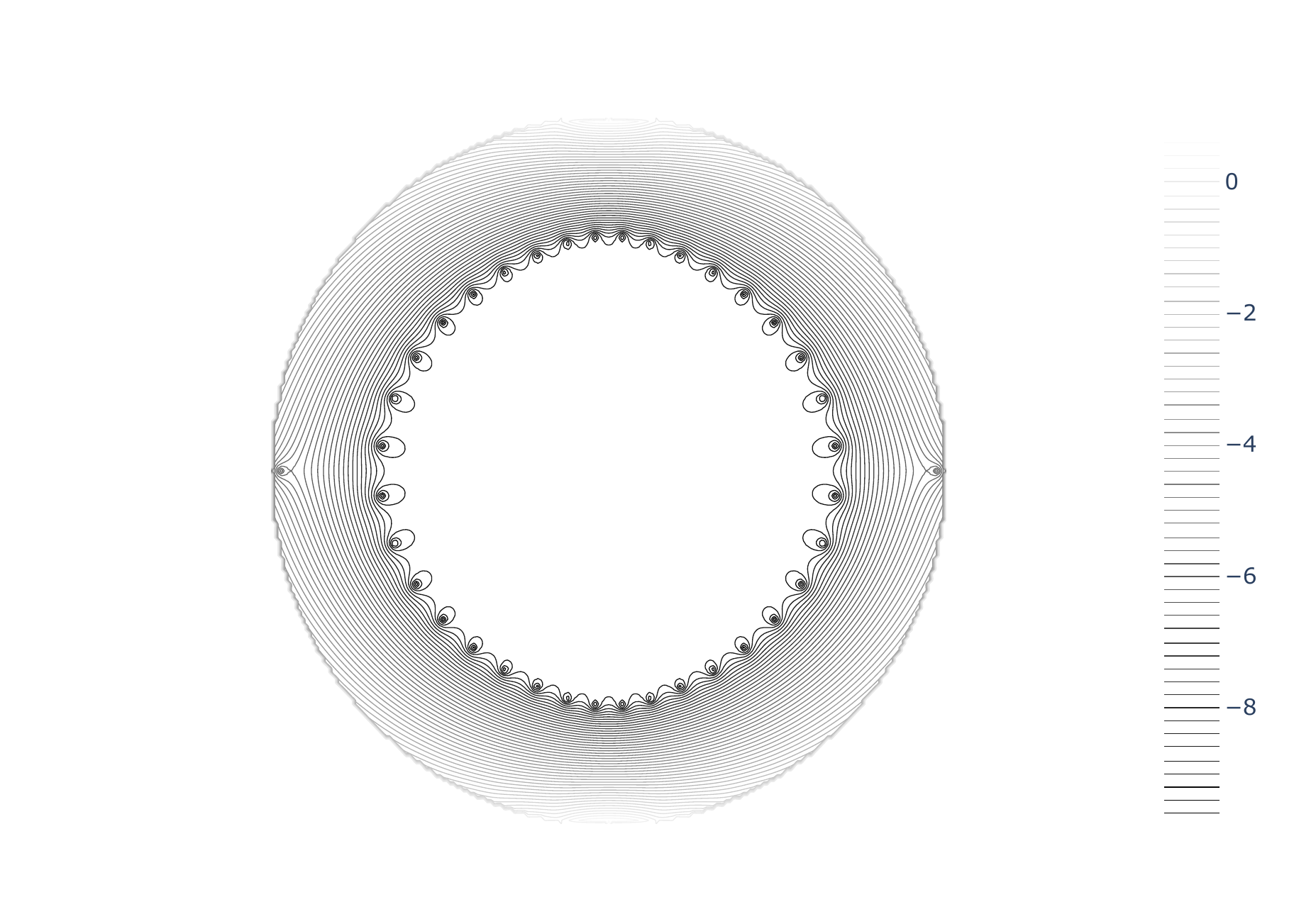}
\includegraphics[scale=0.3]{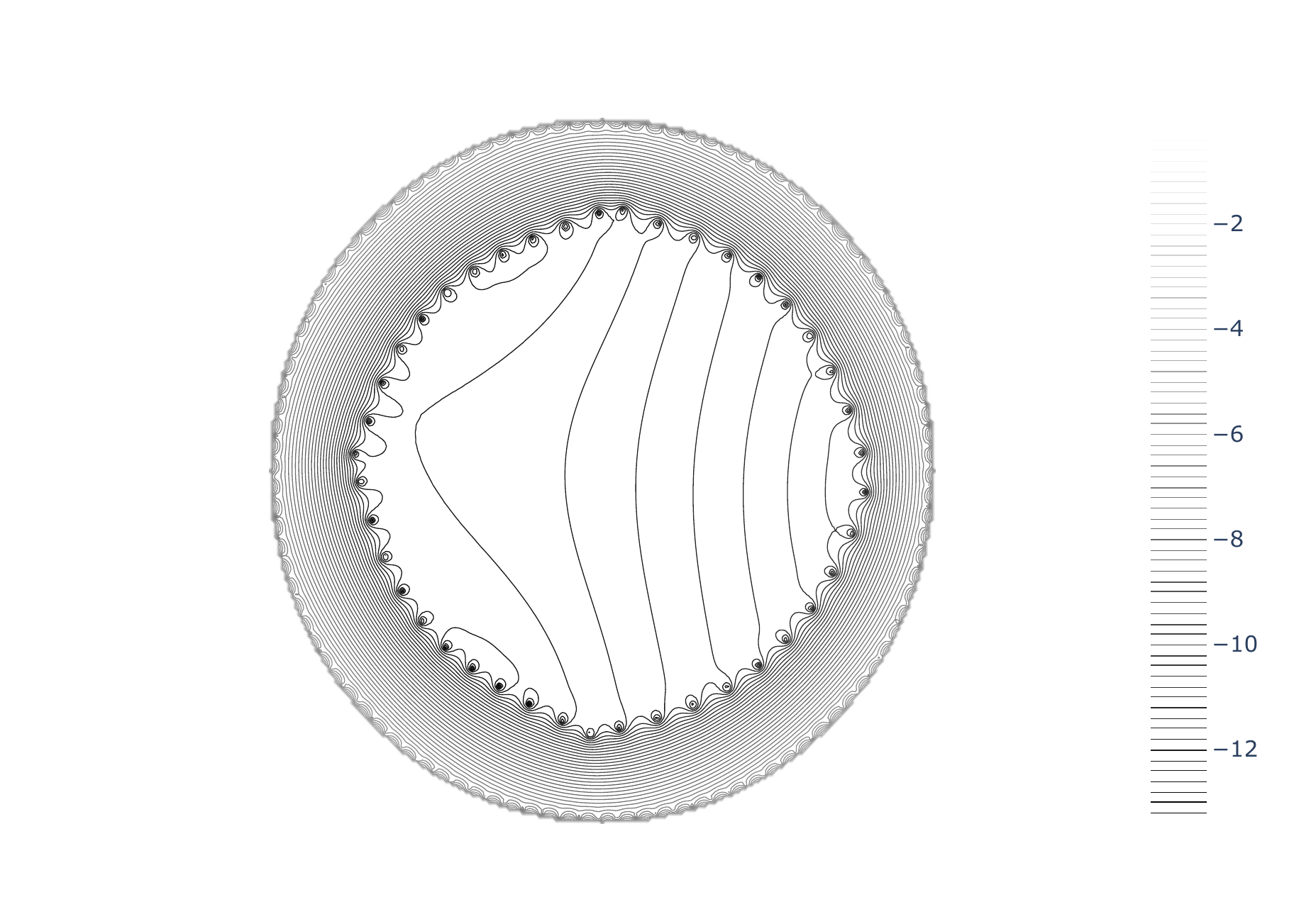}
\includegraphics[scale=0.3]{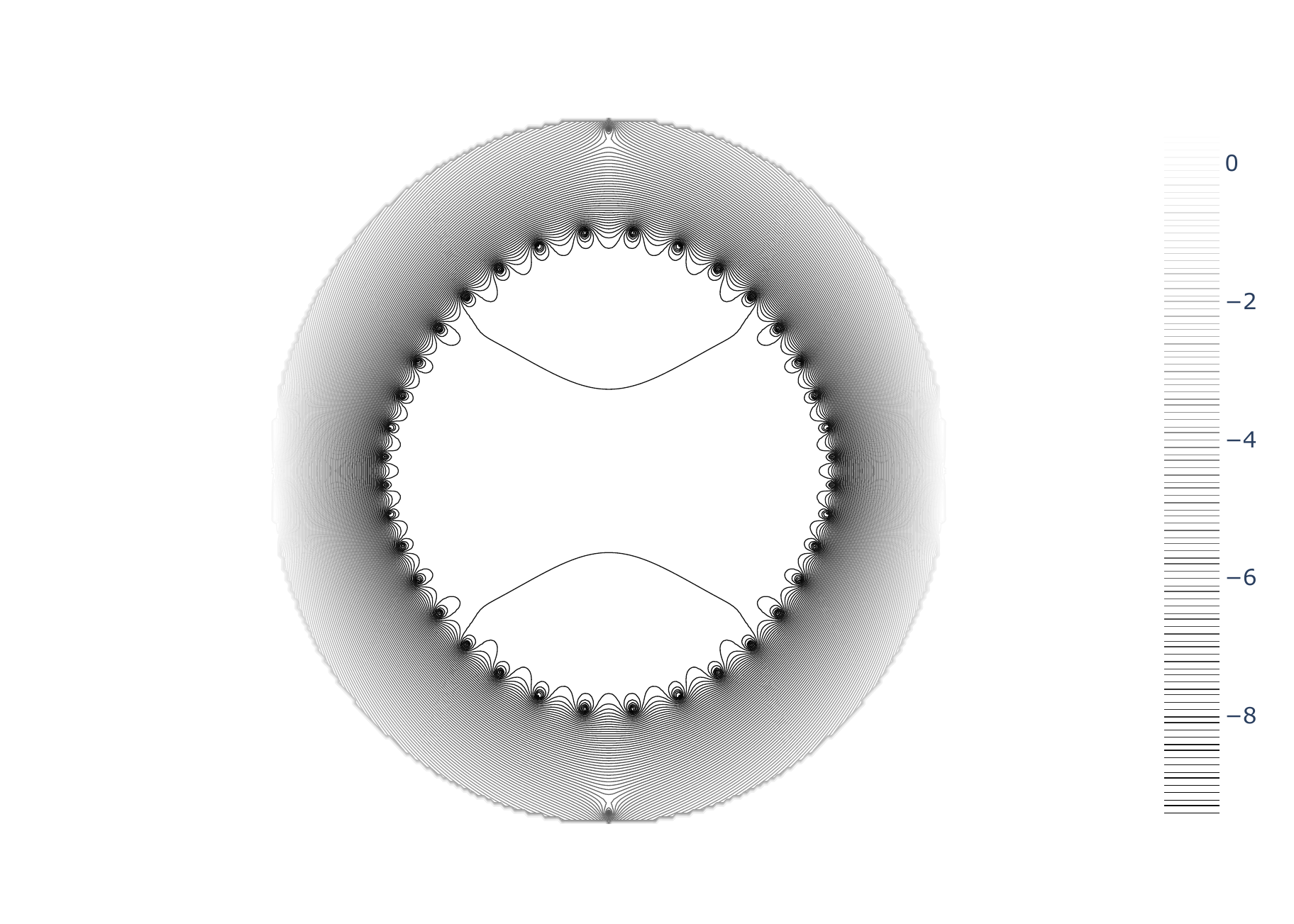}
\end{center}
\caption{Contour plots.}
\label{fig:repr_contour}
\end{figure}

\begin{figure}[htbp]
\begin{center}
\includegraphics[scale=0.25]{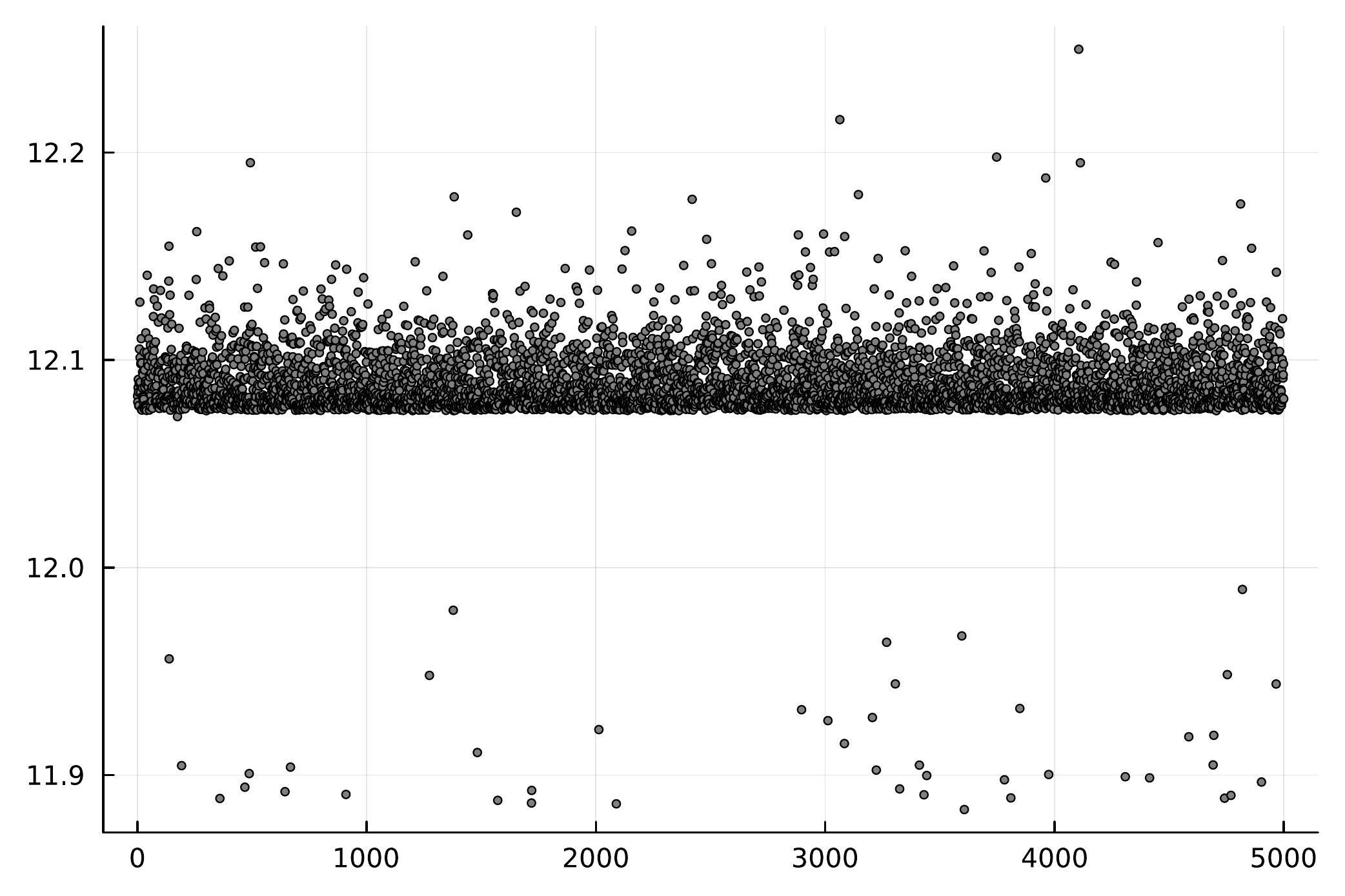}
\includegraphics[scale=0.25]{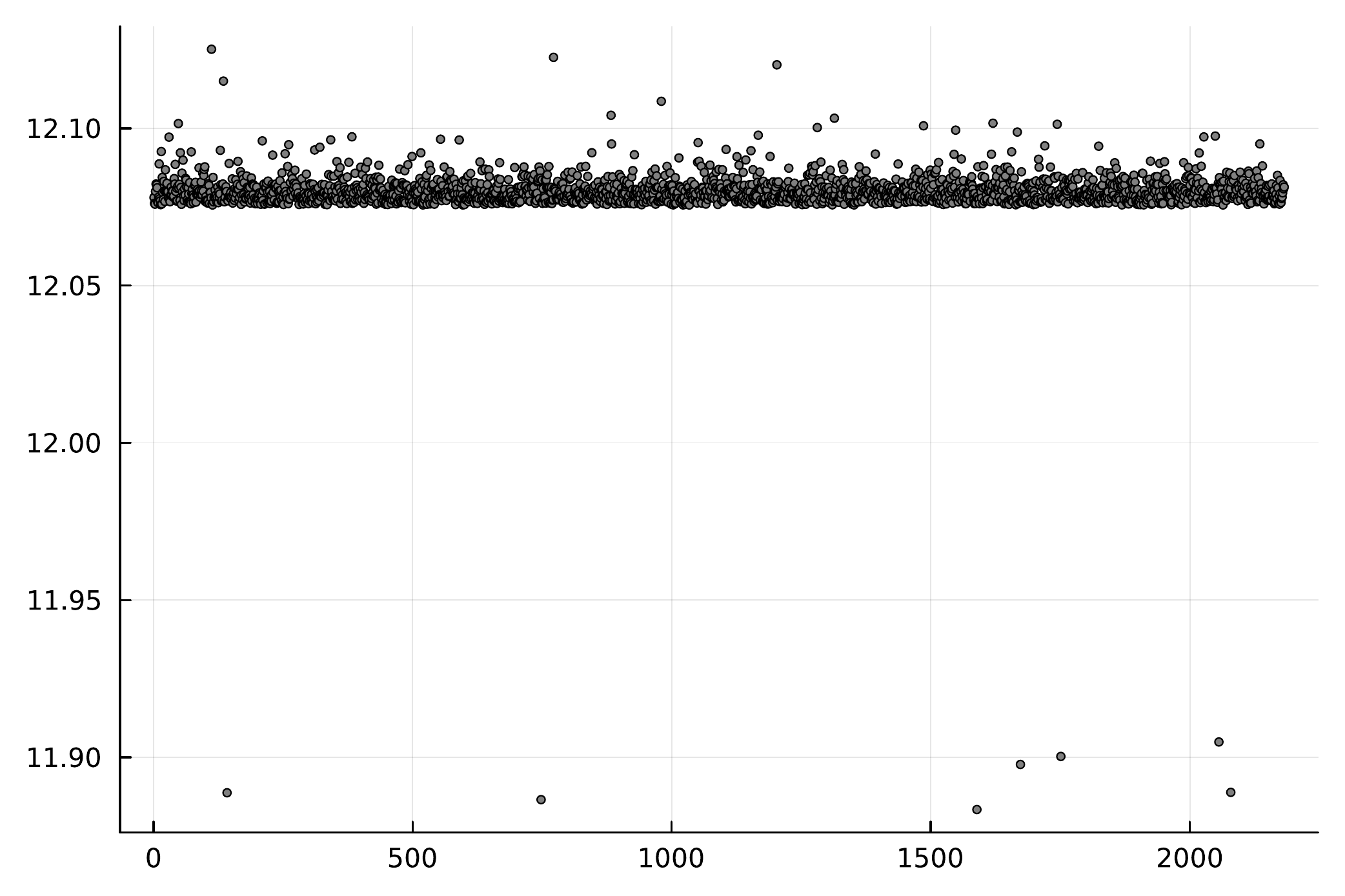}
\includegraphics[scale=0.25]{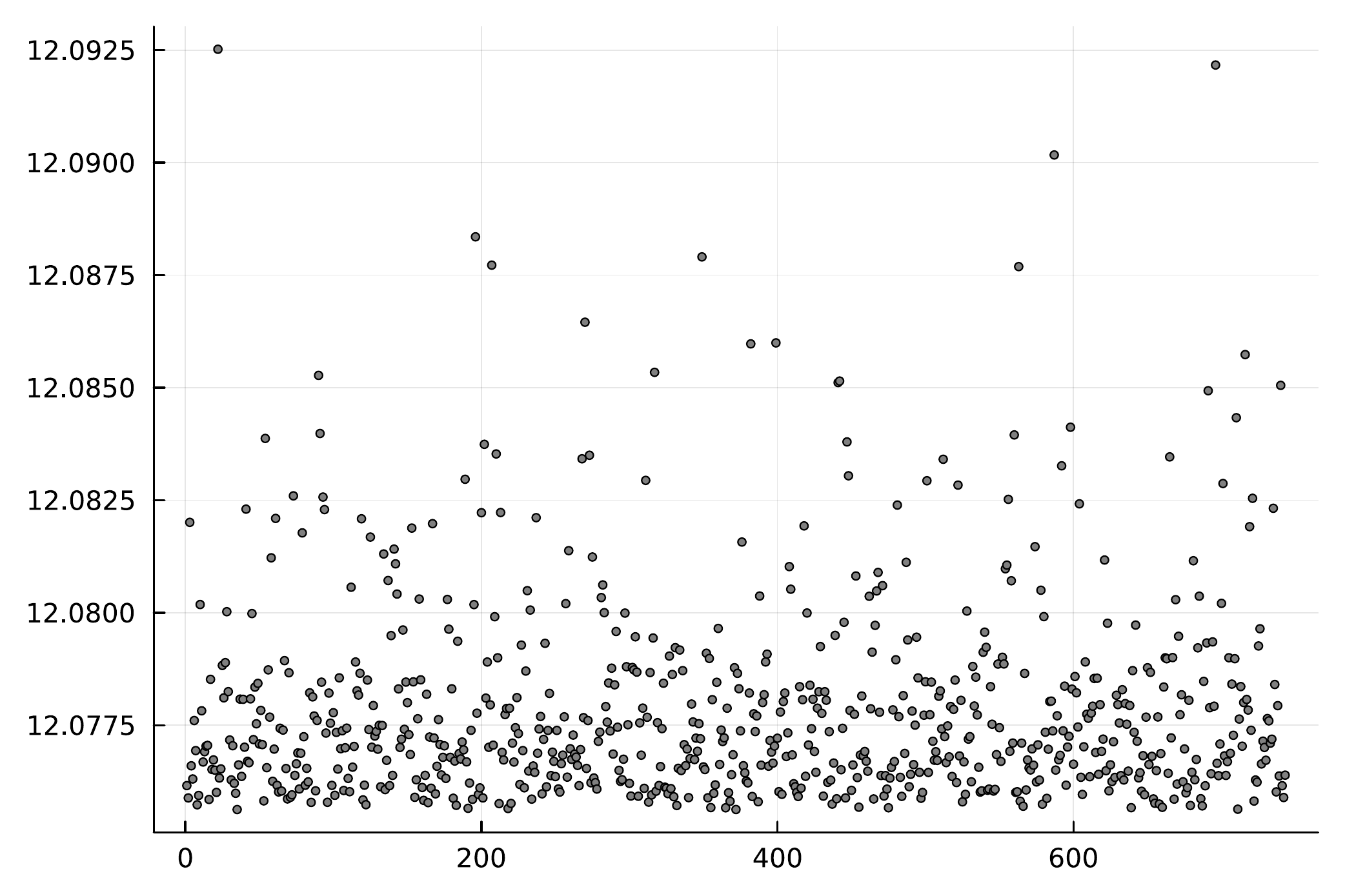}
\end{center}
\caption{Distribution of the Dirichlet energy.}
\label{fig:dst}
\end{figure}

\begin{figure}[htbp]
\begin{center}
\includegraphics[scale=0.25]{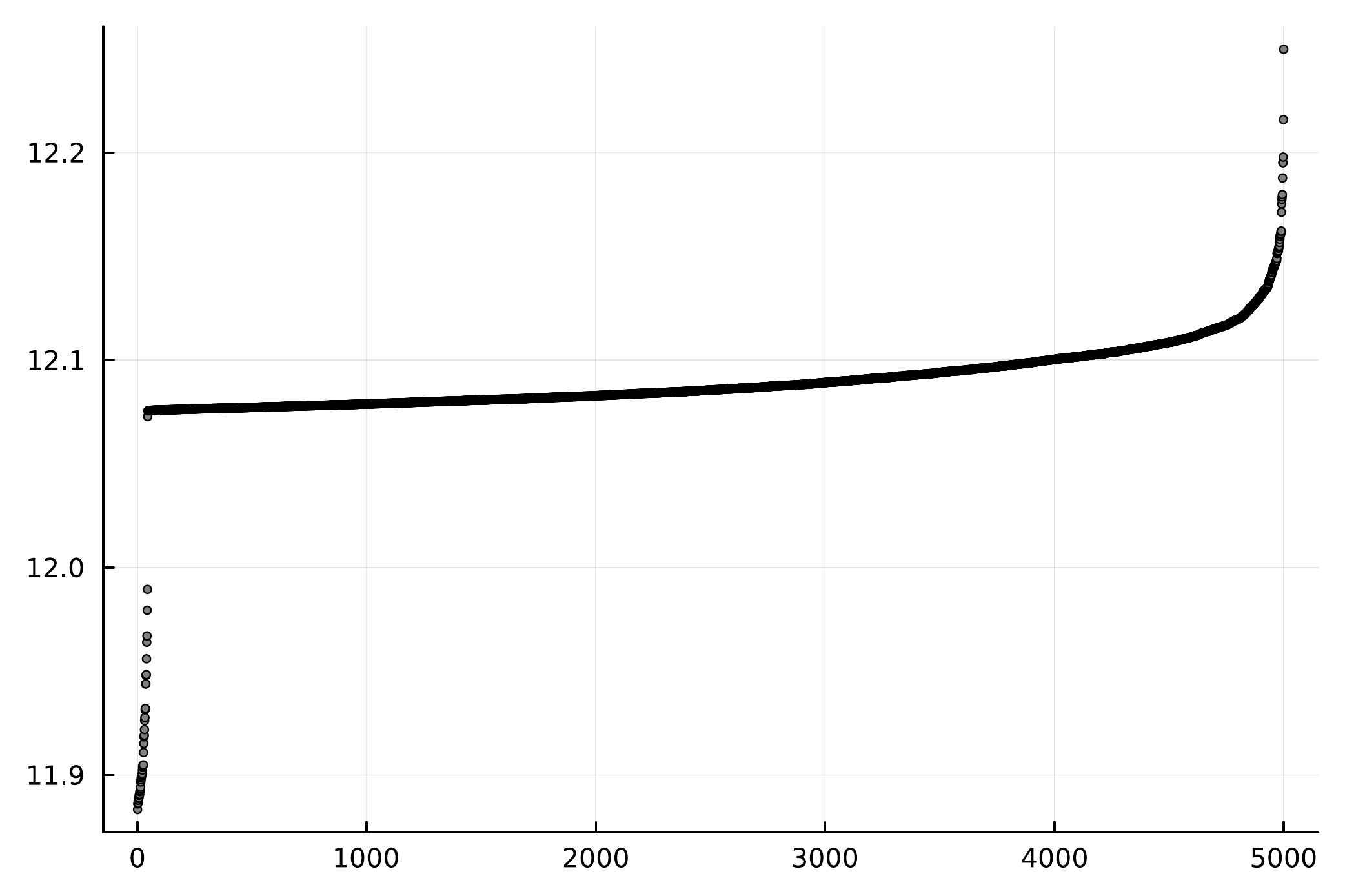}
\includegraphics[scale=0.25]{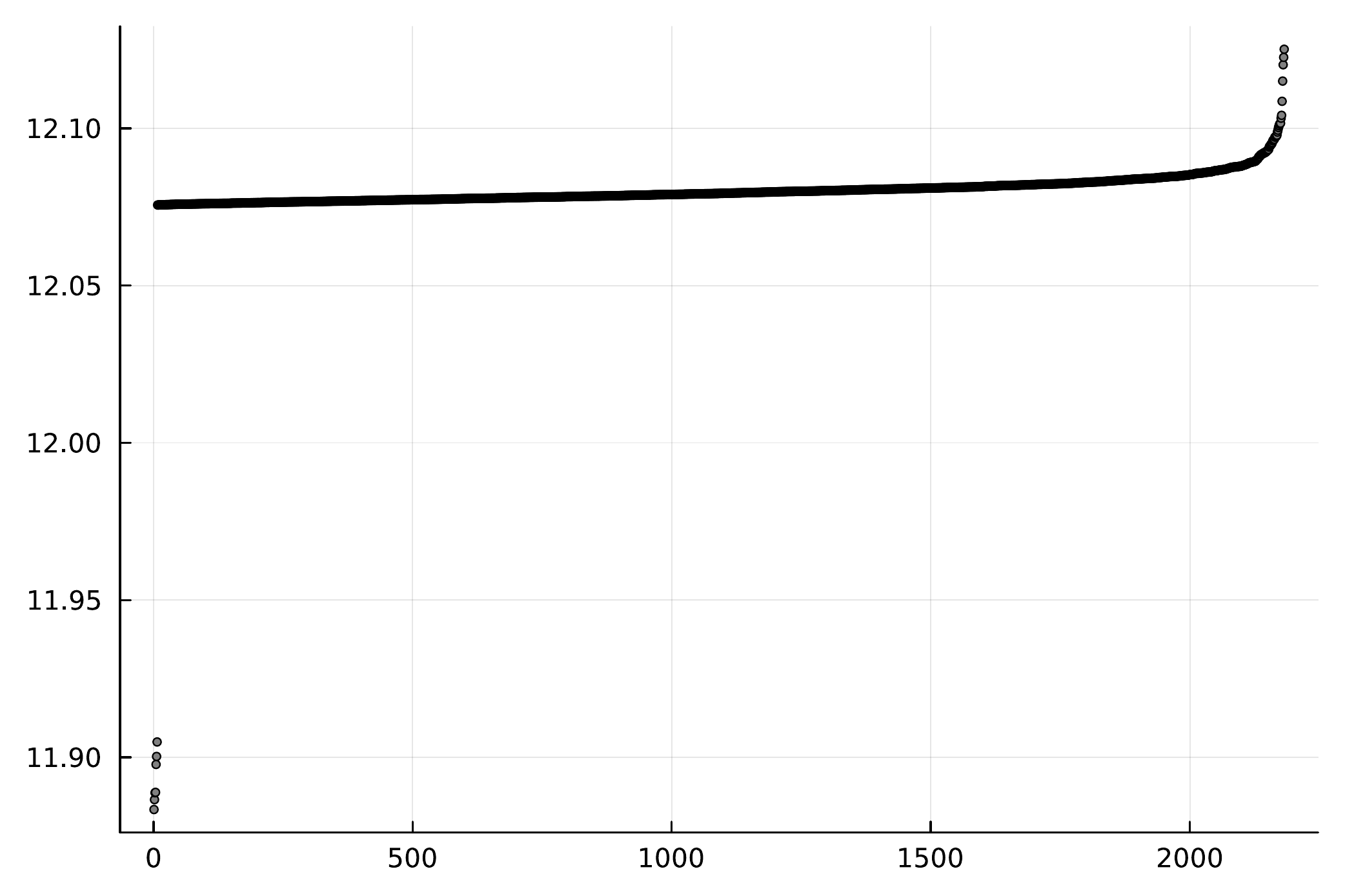}
\includegraphics[scale=0.25]{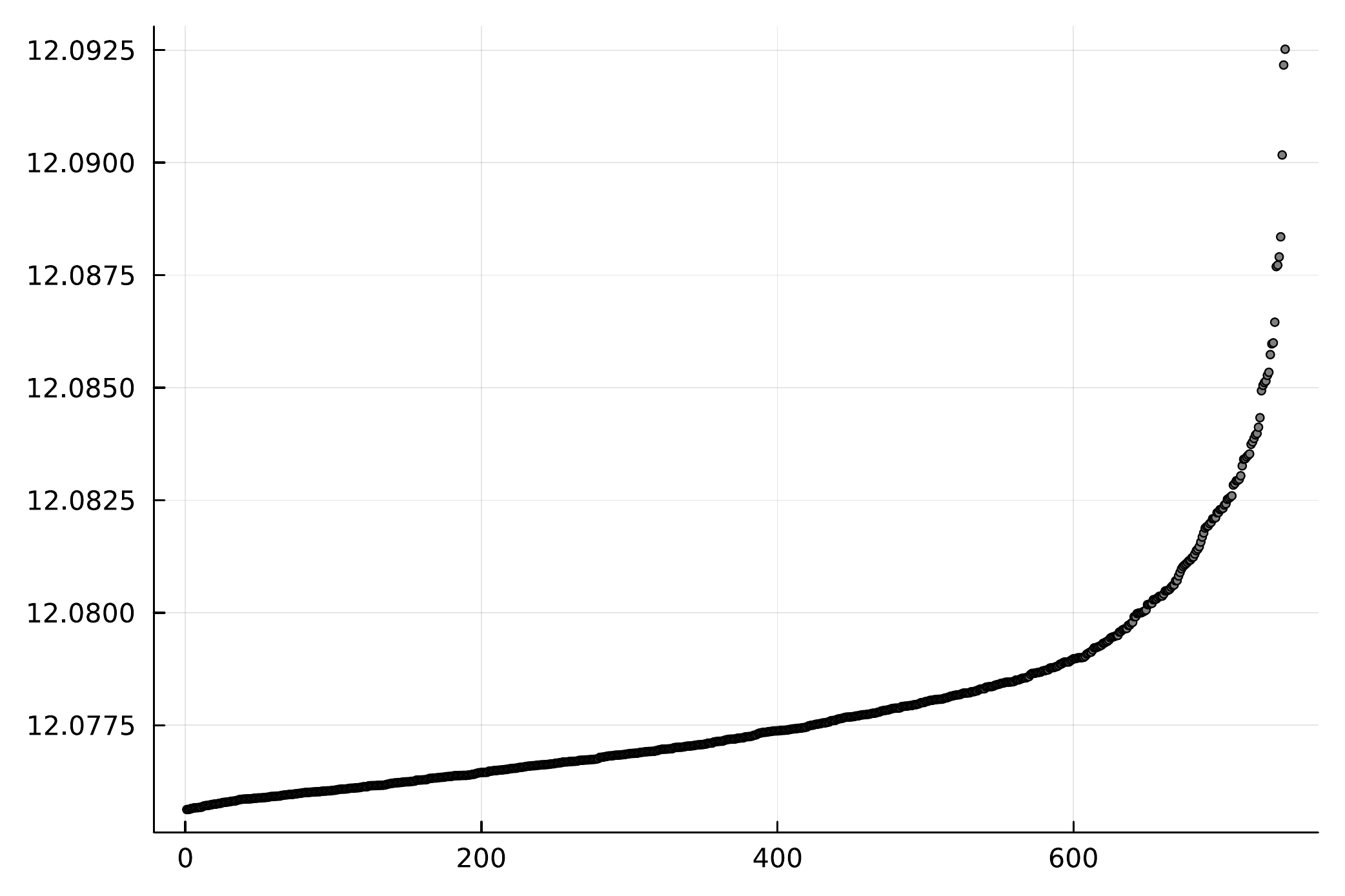}
\end{center}
\caption{Rearranged Distribution of the Dirichlet energy.}
\label{fig:redst}
\end{figure}

\bibliography{BIB2020.bib}

\begin{thebibliography}{10}

\bibitem{amano1991bidirectional}
{\sc K.~Amano}, {\em A bidirectional method for numerical conformal mapping
  based on the charge simulation method}, J. Inform. Process., 14 (1991),
  pp.~473--482.

\bibitem{amano1994charge}
{\sc K.~Amano}, {\em A charge simulation method for the numerical conformal
  mapping of interior, exterior and doubly-connected domains}, J. Comput. Appl.
  Math., 53 (1994), pp.~353--370.

\bibitem{amano1998charge}
{\sc K.~Amano}, {\em A charge simulation method for numerical conformal mapping
  onto circular and radial slit domains}, SIAM J. Sci. Comput., 19 (1998),
  pp.~1169--1187.

\bibitem{amano2012numerical}
{\sc K.~Amano, D.~Okano, H.~Ogata, and M.~Sugihara}, {\em Numerical conformal
  mappings onto the linear slit domain}, Jpn. J. Ind. Appl. Math., 29 (2012),
  pp.~165--186.

\bibitem{BohmeTromba1981}
{\sc R.~B\"{o}hme and A.~J. Tromba}, {\em The index theorem for classical
  minimal surfaces}, Ann. of Math. (2), 113 (1981), pp.~447--499.

\bibitem{Courant_1937}
{\sc R.~Courant}, {\em Plateau's problem and dirichlet's principle}, Ann. of
  Math. (2), 38 (1937), pp.~679--724.

\bibitem{Dierkes_Hildebrandt_Sauvigny_2010}
{\sc U.~Dierkes, S.~Hildebrandt, and F.~Sauvigny}, {\em Minimal surfaces},
  vol.~339 of Grundlehren der Mathematischen Wissenschaften [Fundamental
  Principles of Mathematical Sciences], Springer, Heidelberg, second~ed., 2010.
\newblock With assistance and contributions by A. K\"{u}ster and R. Jakob.

\bibitem{douglas1927method}
{\sc J.~Douglas}, {\em A method of numerical solution of the problem of
  {P}lateau}, Ann. of Math. (2), 29 (1927/28), pp.~180--188.

\bibitem{Douglas1931solution}
\leavevmode\vrule height 2pt depth -1.6pt width 23pt, {\em Solution of the
  problem of {P}lateau}, Trans. Amer. Math. Soc., 33 (1931), pp.~263--321.

\bibitem{dziuk1999discrete-numerics}
{\sc G.~Dziuk and J.~E. Hutchinson}, {\em The discrete {P}lateau problem:
  algorithm and numerics}, Math. Comp., 68 (1999), pp.~1--23.

\bibitem{dziuk1999discrete-convergence}
\leavevmode\vrule height 2pt depth -1.6pt width 23pt, {\em The discrete
  {P}lateau problem: convergence results}, Math. Comp., 68 (1999),
  pp.~519--546.

\bibitem{dziuk2006finite}
\leavevmode\vrule height 2pt depth -1.6pt width 23pt, {\em Finite element
  approximations to surfaces of prescribed variable mean curvature}, Numer.
  Math., 102 (2006), pp.~611--648.

\bibitem{grodet2018finite}
{\sc A.~Grodet and T.~Tsuchiya}, {\em Finite element approximations of minimal
  surfaces: algorithms and mesh refinement}, Jpn. J. Ind. Appl. Math., 35
  (2018), pp.~707--725.

\bibitem{hao2013approximation}
{\sc Y.-X. Hao, C.-J. Li, and R.-H. Wang}, {\em An approximation method based
  on {MRA} for the quasi-{P}lateau problem}, BIT, 53 (2013), pp.~411--442.

\bibitem{hinze1996numerical}
{\sc M.~Hinze}, {\em On the numerical approximation of unstable minimal
  surfaces with polygonal boundaries}, Numer. Math., 73 (1996), pp.~95--118.

\bibitem{katsurada1989mathematical}
{\sc M.~Katsurada}, {\em A mathematical study of the charge simulation method
  ii}, J. Fac. Sci. Univ. Tokyo Sect. IA Math., 36 (1989), pp.~135--162.

\bibitem{katsurada1988mathematical}
{\sc M.~Katsurada and H.~Okamoto}, {\em A mathematical study of the charge
  simulation method i}, J. Fac. Sci. Univ. Tokyo Sect. IA Math., 35 (1988),
  pp.~507--518.

\bibitem{Koiso1983}
{\sc M.~Koiso}, {\em On the finite solvability of {P}lateau's problem for
  extreme curves}, Osaka J. Math., 20 (1983), pp.~177--183.

\bibitem{Lagrange_1760}
{\sc J.~Lagrange}, {\em Essai d{\'u}ne nouvelle m{\'e}thode pour determiner les
  maxima et les minima des formules int{\'e}grales ind{\'e}finies.}, Misc.
  Philos.-Math. Soc. Priv. Taurinensis, 2 (1760--1762), pp.~173--195.

\bibitem{Nitsche1973unique}
{\sc J.~C.~C. Nitsche}, {\em A new uniqueness theorem for minimal surfaces},
  Arch. Rational Mech. Anal., 52 (1973), pp.~319--329.

\bibitem{Nitsche1978finite}
\leavevmode\vrule height 2pt depth -1.6pt width 23pt, {\em Contours bounding at
  most finitely many solutions of {P}lateau's problem}, in Complex analysis and
  its applications ({R}ussian), ``Nauka'', Moscow, 1978, pp.~438--446, 670.

\bibitem{Nitsche1989book}
\leavevmode\vrule height 2pt depth -1.6pt width 23pt, {\em Lectures on minimal
  surfaces. {V}ol. 1}, Cambridge University Press, Cambridge, 1989.
\newblock Introduction, fundamentals, geometry and basic boundary value
  problems, Translated from the German by Jerry M. Feinberg, With a German
  foreword.

\bibitem{pozzi2004estimate}
{\sc P.~Pozzi}, {\em {$L^2$}-estimate for the discrete {P}lateau problem},
  Math. Comp., 73 (2004), pp.~1763--1777.

\bibitem{Rado_1930}
{\sc T.~Rad{\'o}}, {\em Some remarks on the problem of plateau.}, Proc. Natl.
  Acad. Sci. USA, 16 (1930), pp.~242--248.

\bibitem{sakakibara2020bidirectional}
{\sc K.~Sakakibara}, {\em Bidirectional numerical conformal mapping based on
  the dipole simulation method}, Eng. Anal. Bound. Elem., 114 (2020),
  pp.~45--57.

\bibitem{schumacher2019variational}
{\sc H.~Schumacher and M.~Wardetzky}, {\em Variational convergence of discrete
  minimal surfaces}, Numer. Math., 141 (2019), pp.~173--213.

\bibitem{tomek2019discrete}
{\sc L.~Tomek and K.~Mikula}, {\em Discrete duality finite volume method with
  tangential redistribution of points for surfaces evolving by mean curvature},
  ESAIM Math. Model. Numer. Anal., 53 (2019), pp.~1797--1840.

\bibitem{trasdahl2011high}
{\sc O.~y. Tr\aa~sdahl and E.~M. R\o~nquist}, {\em High order numerical
  approximation of minimal surfaces}, J. Comput. Phys., 230 (2011),
  pp.~4795--4810.

\bibitem{tsuchiya1987discrete}
{\sc T.~Tsuchiya}, {\em Discrete solution of the {P}lateau problem and its
  convergence}, Math. Comp., 49 (1987), pp.~157--165.

\bibitem{tsuchiya1990note}
\leavevmode\vrule height 2pt depth -1.6pt width 23pt, {\em A note on discrete
  solutions of the {P}lateau problem}, Math. Comp., 54 (1990), pp.~131--138.

\bibitem{wang2021computing}
{\sc S.~Wang and A.~Chern}, {\em Computing minimal surfaces with differential
  forms}, ACM Trans.\ Graph., 40 (2021).

\end{thebibliography}
\bibliographystyle{siam}

\end{document}